\documentclass[11pt, a4]{amsart}
\usepackage{latexsym,amsmath,amssymb,amsthm,mathrsfs,color,mathtools}
\usepackage{float}
\usepackage{graphicx}
\usepackage{marginnote}
\usepackage{scalerel}
\usepackage{stackengine,wasysym}
\usepackage[bbgreekl]{mathbbol}
\usepackage{bbm}
\usepackage{tikz-cd}
\usetikzlibrary{graphs,decorations.pathmorphing,decorations.markings}
\usepackage[all]{xy}
\usepackage{stmaryrd}
\usepackage{chngcntr}
\usepackage{comment}
\counterwithin{figure}{section}
\usepackage{hyperref}
\hypersetup{colorlinks,%
	citecolor=brown,%
	filecolor=black,%
	linkcolor=blue,%
	urlcolor=black}
\usepackage[toc,page]{appendix}

\usepackage[shortalphabetic]{amsrefs}
\makeatletter
\def\append@label@year@{%
	\safe@set\@tempcnta\bib@year
	\edef\bib@citeyear{\the\@tempcnta}%
	\ifnum\bib@citeyear>9
	\append@to@stem{%
		\ifx\bib@year\@empty
		\else
		\@xp\year@short \bib@citeyear \@nil
		\fi
	}%
	\fi
}
\makeatother

\let\oldtocsection=\tocsection
\renewcommand{\tocsection}[2]{\hspace{0em}\oldtocsection{#1}{#2}}
\usepackage[a4paper]{geometry}

\def\upddots{\mathinner{\mkern 1mu\raise 1pt \hbox{.}\mkern 2mu
		\mkern 2mu \raise 4pt\hbox{.}\mkern 1mu \raise 7pt\vbox {\kern 7
			pt\hbox{.}}} }

\numberwithin{equation}{section}

\begin{document}
	\setlength{\unitlength}{2.5cm}
	\newtheorem{thm}{Theorem}[section]
	\newtheorem{lm}[thm]{Lemma}
	\newtheorem{prop}[thm]{Proposition}
	\newtheorem{cor}[thm]{Corollary}
	\newtheorem{conj}[thm]{Conjecture}
	\newtheorem{specu}[thm]{Speculation}
	
	\theoremstyle{definition}
	\newtheorem{dfn}[thm]{Definition}
	\newtheorem{eg}[thm]{Example}
	\newtheorem{rmk}[thm]{Remark}
	\newcommand{\ome}{\varpi}
	\newcommand{\F}{\mathbf{F}}
	\newcommand{\N}{\mathbbm{N}}
	\newcommand{\R}{\mathbbm{R}}
	\newcommand{\C}{\mathbbm{C}}
	\newcommand{\Z}{\mathbbm{Z}}
	\newcommand{\Q}{\mathbbm{Q}}
	\newcommand{\Mp}{{\rm Mp}}
	\newcommand{\Sp}{{\rm Sp}}
	\newcommand{\GSp}{{\rm GSp}}
	\newcommand{\GL}{{\rm GL}}
	\newcommand{\PGL}{{\rm PGL}}
	\newcommand{\SL}{{\rm SL}}
	\newcommand{\SO}{{\rm SO}}
	\newcommand{\Spin}{{\rm Spin}}
	\newcommand{\GSpin}{{\rm GSpin}}
	\newcommand{\Ind}{{\rm Ind}}
	\newcommand{\Res}{{\rm Res}}
	\newcommand{\Hom}{{\rm Hom}}
	\newcommand{\End}{{\rm End}}
	\newcommand{\msc}[1]{\mathscr{#1}}
	\newcommand{\mfr}[1]{\mathfrak{#1}}
	\newcommand{\mca}[1]{\mathcal{#1}}
	\newcommand{\mbf}[1]{{\bf #1}}
	\newcommand{\mbm}[1]{\mathbbm{#1}}
	\newcommand{\into}{\hookrightarrow}
	\newcommand{\onto}{\twoheadrightarrow}
	\newcommand{\s}{\mathbf{s}}
	\newcommand{\cc}{\mathbf{c}}
	\newcommand{\bfa}{\mathbf{a}}
	\newcommand{\id}{{\rm id}}
	\newcommand{\g}{ \mathbf{g} }
	\newcommand{\w}{\mathbbm{w}}
	\newcommand{\Ftn}{{\sf Ftn}}
	\newcommand{\p}{\mathbf{p}}
	\newcommand{\bq}{\mathbf{q}}
	\newcommand{\WD}{\text{WD}}
	\newcommand{\W}{\text{W}}
	\newcommand{\Wh}{{\rm Wh}}
	\newcommand{\Whc}{{{\rm Wh}_\psi}}
	\newcommand{\ggma}{\omega}
	\newcommand{\sct}{\text{\rm sc}}
	\newcommand{\Of}{\mca{O}^\digamma}
	\newcommand{\gk}{c_{\sf gk}}
	\newcommand{\Irr}{ {\rm Irr} }
	\newcommand{\Irrg}{ {\rm Irr}_{\epsilon} }
	\newcommand{\diag}{{\rm diag}}
	\newcommand{\uchi}{ \underline{\chi} }
	\newcommand{\Tr}{ {\rm Tr} }
	\newcommand{\der}{\de}
	\newcommand{\Stab}{{\rm Stab}}
	\newcommand{\Ker}{{\rm Ker}}
	\newcommand{\bfp}{\mathbf{p}}
	\newcommand{\bfq}{\mathbf{q}}
	\newcommand{\KP}{{\rm KP}}
	\newcommand{\Sav}{{\rm Sav}}
	\newcommand{\de}{{\rm der}}
	\newcommand{\tnu}{{\tilde{\nu}}}
	\newcommand{\lest}{\leqslant}
	\newcommand{\gest}{\geqslant}
	\newcommand{\tu}{\widetilde}
	\newcommand{\tchi}{\tilde{\chi}}
	\newcommand{\tomega}{\tilde{\omega}}
	\newcommand{\Rep}{{\rm Rep}}
	\newcommand{\cu}[1]{\textsc{\underline{#1}}}
	\newcommand{\set}[1]{\left\{#1\right\}}
	\newcommand{\ul}[1]{\underline{#1}}
	\newcommand{\ol}[1]{\overline{#1}}
	\newcommand{\wt}[1]{\widetilde{#1} }
	\newcommand{\wtsf}[1]{\wt{\sf #1}}
	\newcommand{\anga}[1]{{\left\langle #1 \right\rangle}}
	\newcommand{\angb}[2]{{\left\langle #1, #2 \right\rangle}}
	\newcommand{\wm}[1]{\wt{\mbf{#1}}}
	\newcommand{\elt}[1]{\pmb{\big[} #1\pmb{\big]} }
	\newcommand{\ceil}[1]{\left\lceil #1 \right\rceil}
	\newcommand{\floor}[1]{\left\lfloor #1 \right\rfloor}
	\newcommand{\val}[1]{\left| #1 \right|}
	\newcommand{\aff}{ {\rm aff} }
	\newcommand{\ex}{ {\rm ex} }
	\newcommand{\exc}{ {\rm exc} }
	\newcommand{\HH}{ \mca{H} }
	\newcommand{\HKP}{ {\rm HKP} }
	\newcommand{\std}{ {\rm std} }
	\newcommand{\motimes}{\text{\raisebox{0.25ex}{\scalebox{0.8}{$\bigotimes$}}}}
	\newcommand{\rest}{\lvert}
	\newcommand{\ind}{ {\rm ind} }
	\newcommand{\Cusp}{{\rm Cusp}}
	\newcommand{\soc}{{\rm soc}}
	
	\newcommand{\WF}{{\rm WF}}
	\newcommand{\AZ}{{\rm AZ}}

	\newcommand{\FF}{\mca{F}}
	\newcommand{\tv}{\tilde{v}}
	\newcommand{\betaa}{\pmb{\beta}}
	\newcommand{\deltaa}{\pmb{\delta}}
	\newcommand{\bepsilon}{\bar{\epsilon}}
	
	\newcommand{\red}[1]{\textcolor{red}{#1}} 
	\newcommand{\blue}[1]{\textcolor{blue}{#1}}

	\title[Wavefront sets for genuine representations of $\rm GL$-covers]{Wavefront sets for genuine representations of $\rm GL$-covers of Kazhdan--Patterson or Savin types}

	\author{Fan Gao}
	\address{School of Mathematical Sciences, Zhejiang University, 866 Yuhangtang Road, Hangzhou,
		China 310058}
	\email{gaofan@zju.edu.cn}
	
	\author{Runze Wang}
	\address{Department of Mathematics, Ben-Gurion University of the Negev, Be’er Sheva, Israel 8410501}
	\email{mathrunze@gmail.com}
	
	\author{Jiandi Zou}
	\address{Institute for Advanced Study in Mathematics of Harbin Institute of Technology, Harbin, China}
	\email{idealzjd@gmail.com}

	\date{}
	\maketitle
	
	\begin{abstract} 
		
		First, we consider general Brylinski--Deligne covers of the $p$-adic general linear groups, and discuss the theory of Bernstein--Zelevinsky derivatives. We also recall the Zelevinsky-type classification of the irreducible  genuine spectrum for the Kazhdan--Patterson and Savin covers. Following this, for these two special families of covers, we determine the wavefront sets of their irreducible genuine representations, expressed in terms of the iterated degrees of the highest Bernstein--Zelevinsky derivatives. Finally, for Kazhdan--Patterson covers, we reinterpret this result on the wavefront set using a version of the local Langlands correspondence and the covering Barbasch--Vogan duality.
		
	\end{abstract}
	
	\tableofcontents
	
	\section{Introduction}\label{Intro}
	
	\subsection{Background}
	
	Let $F$ be a finite extension of $\Q_p$. %(In fact, we can also take $F$ to be a local field of sufficiently large characteristic $p$.)
	We consider the general linear group $\mbf{G}_r:=\GL_r$ and its $F$-points $G_r=\mbf{G}_r(F)$. The theory of Bernstein--Zelevinsky derivatives, as developed in \cite{bernstein1976representations,bernstein1977induced, zelevinsky1980induced}, has been an indispensable tool in understanding irreducible admissible representations of $G_r$. Coupled with the Zelevinsky classification of the irreducible spectrum $\Irr(G_r)$ in terms of the multisegments, the theory provides many insights on questions regarding the structure of induced representations, the branching problem, the degenerate Whittaker models, etc., to name a few; see \cite{lapid2016parabolic, chan2022restriction, gurevich2021restriction, gomez2017generalized} and references therein.
	
	For any $\pi \in \Irr(G_r)$, one important representation-theoretic invariant associated with $\pi$ is its wavefront set ${\rm WF}(\pi)$. It gives the Gelfand--Kirillov dimension of $\pi$, which measures the ``size'' of $\pi$ in an asymptotic way, see for instance \cite{Sav94, dalal2025uniform}. The intimate relation between the theory of Bernstein--Zelevinsky derivatives and that of Whittaker models enables one to determine ${\rm WF}(\pi)$ in terms of the multisegment structure $\mfr{m}(\pi)$ underlying $\pi$ (i.e., the Zelevinsky classification of $\pi$), and one has 
	\begin{equation} \label{E:WF1}
		\WF(\pi)=\set{\lambda_{\mfr{m}(\pi)}},
	\end{equation}
	where $\lambda_{\mfr{m}(\pi)}$ is a partition of $r$ arising from the (iterated) Bernstein--Zelevinsky derivatives of $\pi$. On the other hand, in view of the local Langlands correspondence for $G_r$, one has
	\begin{equation} \label{E:WF2}
		{\rm WF}(\pi) = \set{d_{BV,G_r}(\mca{O}(\phi_{\AZ(\pi)}))}.
	\end{equation}
	Here, $\AZ(\pi)$ is the Aubert--Zelevinsky dual of $\pi$, $\phi_{\AZ(\pi)}$ is its L-parameter, and $d_{BV,G_r}(\mfr{p})=\mfr{p}^\top$ is the Barbasch--Vogan duality (just the transpose map for $G_r$). The equalities \eqref{E:WF1} and \eqref{E:WF2} were proved in \cite{moeglin1987modeles}. In fact, for general linear algebraic group $G$, there has been an abundance of literature regarding $\WF(\pi)$ for $\pi\in \Irr(G)$, see for example \cite{moeglin1987modeles, gomez2017generalized, GGS21,HLLS,CK,CMBO,CMBO24,CMBO25,JL16,JL25}  and the references therein.

	In this paper, we escalate to consider degree-$n$ central covering group $\ol{G_r}$ of $G_r$, which sits in a short exact sequence
	$$\begin{tikzcd}
		\mu_n \ar[r, hook] & \overline{G_r} \ar[r, two heads] & G_r
	\end{tikzcd}$$ 
	and arises from the Brylinski--Deligne (BD) framework \cite{BD01}. We call $\pi \in \Irrg(\ol{G_r})$ an irreducible ($\epsilon$-)genuine representation if $\mu_n$ acts via a fixed embedding $\epsilon:\mu_n \into \C^\times$. We are motivated from \eqref{E:WF1} and \eqref{E:WF2} above, and the goal is to establish similar equalities for $\pi\in \Irrg(\ol{G_r})$. Note that this is not a task to be accomplished mutatis mutandis. Indeed, it is well-known that genuine representations of $\ol{G_r}$ tend to possess larger wavefront set and thus bigger Gelfand--Kirillov dimension. This implies in particular that the left-hand sides of \eqref{E:WF1} and \eqref{E:WF2} tend to be bigger for genuine $\pi$ than its linear counterpart. 
	
	We elaborate on the issues involved and also state our main result below.
	
	\subsection{Main result}
	First, the basic mechanism and principles underlying the Bernstein--Zelevinsky derivatives can be applied to any BD covering group $\ol{G_r}$. However, for a general BD covering group of $\ol{G_r}$, the potential non-commutativity of blocks of a covering Levi subgroup posits technical issues with such an application. Thus, starting from Section \ref{section resultKPS} we restrict to two of the most natural and well-studied central extensions of $G_r$: the Savin (S) covers, which arise from the restriction of covers of $\text{Sp}_{2r}$ to their Siegel Levi, and the Kazhdan--Patterson (KP) covers, which are the most studied covers with many applications, such as \cite{KP84, BG92, BBF11, Tak14, kaplan2022rankin, kaplan2019doubling, zou2025simple}. 
	From the technical perspective of this paper, for Savin covers of $G_r$, the blocks of a covering Levi subgroup do commute from its structure de facto; for KP covers, one can utilize the metaplectic tensor product construction (see \cite{Mez04, takeda2016metaplectic, kaplan2022classification}). In fact, we consider certain extended metaplectic tensor product in this paper, which seems to be more convenient and robust in order to adapt the Bernstein--Zelevinsky theory for the Savin covers or Kazhdan--Patterson covers.

	For $\pi \in \Irrg(\ol{G_r})$, in order to compute $\WF(\pi)$ by determining $\lambda_{\mfr{m}(\pi)}$ as in the linear case for $G_r$, it is certainly a prerequisite to have a Zelevinsky-type classification for $\Irrg(\ol{G_r})$. Such a classification is given in the earlier work of Kaplan--Lapid--Zou \cite{kaplan2022classification} for Kazhdan--Patterson covers. It was already noted in ibid. that the classification holds for Savin covers as well. We mention in passing that for general BD cover of $G_r$, a Zelevinsky-type construction is not sufficient to describe the whole of $\Irrg(\ol{G_r})$. It is expected that there exist essential extra difficulties besides the non-commutativity issue. For instance, the parabolic induction of a unitary cuspidal representation might not be irreducible anymore for general BD cover of $G_r$. In any case, this entails a second reason why we restrict to the Savin covers and Kazhdan--Patterson covers, for which one has $\pi = Z(\mfr{m})$ with $\mfr{m}:=\mfr{m}(\pi)$ a multisegment for $\ol{G_r}$. 
	
	For $\pi \in \Irrg(\ol{G_r})$, we write $\pi^{(k)}$ for the highest Bernstein--Zelevinsky derivative of $\pi$, for some $k\in \N_{\gest 1}$. Then we obtain iteratively a partition $\lambda_\pi:=(k_1, k_2, ..., k_s)$ arising from the iterated highest Bernstein--Zelevinsky derivatives, see \S \ref{section HighDegWhit} for a detailed description. For Savin covers or Kazhdan--Patterson covers, one has $\pi=Z(\mfr{m})$ as mentioned above, and it follows from Theorem \ref{thm highestderivative} that $\lambda_\pi$ can be computed from $\mfr{m}$ in an explicit way; thus we write $\lambda_{\mfr{m}}:=\lambda_\pi$. This also gives us our main result of the paper, which is an extension of \eqref{E:WF1} in \cite{moeglin1987modeles} :

	\begin{thm}[Theorem \ref{T:WF}] \label{T:01}
		Let $\ol{G_r}$ be an $n$-fold cover of $G_r$ of Savin type or Kazhdan--Patterson type. Assume $p\nmid n$. Then for every $Z(\mfr{m}) \in \Irrg(\ol{G_r})$, one has
		$$\WF(Z(\mfr{m})) = \set{\lambda_\mfr{m}}.$$ 
	\end{thm}
	
	In particular, the wavefront set in this case is also a singleton. Also, in Theorem \ref{thm genericclassify}, we give a classification of the generic genuine irreducible representations of these $\ol{G_r}$. Moreover, by the definition of $\lambda_{\mfr{m}}$ and the computation of $l(\rho)$ in \cite{zou2025simple}*{\S 6.3, \S 9.5}, we see that for Savin or Kazhdan--Patterson covers, if $n^* \gest r$, then every $\pi \in \Irrg(\ol{G_r})$ is generic, i.e., it possesses Whittaker models. Here $n^*=n$ for KP-covers and $n^*=n/\gcd(n,2)$ for S-covers. We also note that the equality in Theorem \ref{T:01} has been verified for Iwahori-spherical representations with positive Satake parameter, for arbitrary BD cover $\ol{G_r}$, see \cite{GW}.
	
	Clearly, an extension of \eqref{E:WF2} to covering $\ol{G_r}$ is contingent on its local Langlands correspondence, which has not been established in general. Moreover, one needs to replace $d_{BV,G_r}$ by the covering Barbasch--Vogan duality $d_{BV, G_r}^{(n)}$, defined and studied in \cite{GLLS}.
	In the last part of this paper, we try to reconcile Theorem \ref{T:01} with a form generalizing \eqref{E:WF2} to covers, using $d_{BV, G_r}^{(n)}$.
	For this purpose, we further confine to Kazhdan--Patterson cover, and in this case we can define the L-parameter $\tilde{\phi}_\pi$ of any $\pi \in \Irrg(\ol{G_r})$ by utilizing the metaplectic correspondence between a KP-cover $\ol{G_r}$ and $G_r$ established for discrete series in \cite{flicker1986metaplectic} (see also \cite{zou2022metaplectic}), and the Zelevinsky classification for $\ol{G_r}$ as in \cite{kaplan2022classification}. Granted with these, we show in Corollary \ref{cor WFZmBVdual} that 
	$$\WF(Z(\mfr{m})) = \set{d_{BV, G_r}^{(n)}(\mca{O}(\tilde{\phi}_{L(\mfr{m})}))},$$
	where $L(\mfr{m})$ is just the Aubert--Zelevinsky dual of $Z(\mfr{m})$. This is the sought generalization of \eqref{E:WF2} to $\ol{G_r}$ mentioned above.
	
	% and is also an improvement of \cite[Theorem 1.3]{GLLS}.

	\subsection{Acknowledgement} We would like to thank the anonymous referee for some very insightful and helpful comments which help improve the paper at several places. 
	
	The work of F. Gao and R. Wang is partially supported by the National Key R\&D Program of China (No. 2022YFA1005300) and also by NSFC-12171422, and the work of J. Zou is partially supported by the start-up funding of HIT (No. AUGA5710010825).

	\section{Extended metaplectic tensor product and Whittaker models}
	
	\subsection{Covers of $G_r$}
	
	Let $F$ be a non-archimedean local field with discrete valuation $|\cdot|_F$ and residual characteristic $p$. Starting from \S \ref{section WFandLLC}, we further require that $F$ is of characteristic 0. We fix a positive integer $n$, such that the subgroup of $n$-th roots of unity in $F^{\times}$, denoted by $\mu_n$, is of cardinality $n$. Starting from \S \ref{section resultKPS}, we assume that we are in the tame case, meaning that $n$ is coprime to $p$.

	\subsubsection{Finite central covers}
	
	By an $n$-fold cover of an $\ell$-group $G$, we mean a central extension 
	$$\begin{tikzcd}
		\mu_n \ar[r, hook] & \overline{G} \ar[r, two heads] & G
	\end{tikzcd}$$ 
	of $G$ by $\mu_n$ as topological groups. Let $\mbf{p}:\ol{G}\rightarrow G$ be the natural projection. A section $\mbf s:G\rightarrow \ol{G}$ is a continuous map such that $\mbf p\circ \mbf s=\id$. 
	
	All representations in this article are smooth complex representations. We fix a faithful character $\epsilon:\mu_n\rightarrow \C^{\times}$. For an $n$-fold cover $\ol{H}$ of an $\ell$-group $H$, we consider ($\epsilon$-)genuine representations of $\ol{H}$, meaning that those representations with the $\mu_n$-action given by the multiplication by $\epsilon$. For an $\epsilon$-genuine representation $\rho$ of $\ol{H}$, its contragredient $\rho^{\vee}$ is an $\epsilon^{-1}$-genuine representation.
	
	The commutator $\ol{G}\times \ol{G}\rightarrow \ol{G}$ factors through $G\times G$, and we denote by 
	$$[\cdot, \cdot]:G\times G\rightarrow \ol{G}$$
	the resulting map. In particular, given two commuting elements $h_1,h_2$ in $G$, we have $[h_1,h_2]= \mbf s(h_1)\mbf s(h_2)\mbf s(h_1)^{-1}\mbf s(h_2)^{-1}\in\mu_n$, where $\mbf s$ is any section of $G$. 
	
	A splitting of a closed subgroup $H$ of $G$ is a section as well as a group homomorphism $\mbf s:H\rightarrow\ol{H}$. Once we fix a (sometimes, canonical) choice of $\mbf s$, we identify $H$ with the image $\mbf s(H)$ as a subgroup of $\ol{H}$. One typical case is that when $\gcd(p,n)=1$, there exists a unique splitting of a pro-$p$-subgroup $H$, thus we may canonically regard $H$ as a subgroup of $\ol{G}$. Another typical example is that for a reductive group $G$, there exists a unique section $\mbf s$ on the set of unipotent elements of $G$ such that it is invariant up to $G$-conjugacy and its restriction to any unipotent subgroup $N$ of $G$ is a splitting. Using $\mbf{s}$, we canonically identify any unipotent subgroup $N$ of $G$ as a subgroup of $\ol{G}$.
	
	\subsubsection{Finite central covers of $G_r$}
	
	For a positive integer $r$, the group $G_r=\GL_r(F)$ is an $\ell$-group. Denote $\nu:=\val{\det(\cdot)}_F: G_r \to \C^\times$.
	
	By an $n$-fold Brylinski--Deligne cover, we mean a central extension studied in \cite{BD01} associated with a certain pair $(D,\eta)$. See also \cite{weissman2018,gan2018langlands} for more references. We write BD-cover for short. 
	For $G_r$, its $n$-fold BD covers are essentially parametrized by a Weyl-invariant bilinear form $B_Q: Y\times Y \to \Z$, where $Y$ is the cocharacter lattice of $G_r$ with a standard basis $\{e_i:1\lest i \lest r\}$.
	The Weyl-invariance enforces that $B_{Q}$ is determined by $\mathbf{a},\mathbf{b}\in \Z$ such that 
	$$B_Q(e_i,e_j)=\begin{cases}
		2\mathbf{a} & i=j, \\ \mathbf{b} & i \neq j.
	\end{cases}$$
	We have $Q(\alpha^\vee)=2\mbf{a}-\mbf{b}$ for any coroot $\alpha^\vee$, and write
	$$n_\alpha:=\frac{n}{\gcd(n,Q(\alpha^\vee))}.$$
	
	A Kazhdan--Patterson cover is a BD-cover of $G_r$ such that $Q(\alpha^\vee)=-1$ and a Savin cover is a BD-cover of $G_r$ associated with the pair $(\textbf{a}=-1,\textbf{b}=0)$. We write KP-cover or S-cover for short. It is clear that a KP-cover is determined by the number $\mbf  a \in\Z$, which corresponds to the notation $c$ used in \cite{KP84}. For a KP-cover, we write
	\begin{equation}\label{eq drgcd}
		d_r:=\gcd(n,2r\mbf{a}-r+1)	.
	\end{equation}
	
	We note that the BD-cover $\ol{G_r}$ associated with $\mbf{a}, \mbf{b}$ has the following heredity property. If we take $G_{r'}, r'\lest r$ and view it as a subgroup of $G_r$ in the natural way (i.e., as a block of a standard Levi subgroup of  $G_r$), then the cover $\ol{G_{r'}}$ obtained as the restriction from $\ol{G_r}$ is also associated with $\mbf{a}, \mbf{b}$. By convention, we have $\ol{G_{0}}:=\mu_n$. For a closed subgroup $H \subseteq G_r$, we write
	$$H^{(n)}:=\set{h\in H: \ \det(h) \in F^{\times n}} \subseteq H$$
	for the subgroup consisting of elements whose determinants are always a power of $n$.

	\subsection{Metaplectic tensor product}
	
	Let $M=G_{r_1}\times\dots\times G_{r_k}$ be a Levi subgroup of $G_r$, where $r=r_1+\dots+r_k$. Write $\beta=(r_1,\dots,r_k)$ for the related composition of $r$. 
	Let $H=H_{1}\times\dots\times H_{k}$ be a subgroup of $M$, such that each $H_{i}$ is a closed subgroup of $G_{r_{i}}$. We call $H$ \emph{block compatible} if $[H_{i},H_{j}]=\{1\}$ for any  $1\lest i<j\lest k$. For example, $H$ is block compatible in the following two cases:
	\begin{itemize}
		\item[(i)] for every $i$, the determinant of every element in $H_{i}$ is an $n$-th power in $F^{\times}$;
		\item[(ii)] $\gcd(n,p)=1$, and for every $i$ the determinant of every element in $H_{i}$ is in $\mfr o _{F}^{\times}$.
	\end{itemize}
	Let $\rho_{i}$ be a genuine representation of $\ol{H_{i}}$ for each $i$. We take the tensor product $\rho_{1}\boxtimes\dots\boxtimes\rho_{k}$ as a representation of $\ol{H_{1}}\times\dots\times \ol{H_{k}}$, trivial on $$\Xi=\{(\zeta_{1},\dots,\zeta_{k})\in\mu_{n}\times\dots\times\mu_{n}\mid \zeta_{1}\dots\zeta_{k}=1\}.$$
	If $H$ is block compatible, then we have the isomorphism 
	$$\ol{H}\simeq \ol{H_{1}}\times\dots\times \ol{H_{k}}/\Xi,$$
	and in this case $\rho_{1}\boxtimes\dots\boxtimes\rho_{k}$ gives rise to a genuine representation of $\ol{H}$, and every irreducible genuine representation of $\ol{H}$ is of this form.
	
	\subsubsection{Tensor product for S-covers}
	
	If $\ol{G_r}$ is an S-cover, then any subgroup $H$ as above is block compatible. This is because the parameter $\mbf{b}=0$ for Savin covers, and thus $[G_{r_i}, G_{r_j}]=1$ for $i \ne j$. Thus we may always take the tensor product of genuine representions of $\ol{H_i}$ and get a representation of $\ol{H}$. In particular, if each representation of $\ol{H_i}$ is of finite length (resp. irreducible), so is the tensor product for $\ol{H}$.
	
	\subsubsection{Metaplectic tensor product for KP-covers}
	
	If $\ol{G_r}$ is a KP-cover, then things become slightly more complicated, since KP covers are not block-compatible. However, given a locally finite genuine representation of a Levi subgroup of $\ol{G_{r_i}}$ with fixed ``central character" for each $i$, there is still a way of defining the so-called ``metaplectic tensor product'', once a compatible ``central character" of $\ol{G_r}$ is fixed (\emph{cf.} \cite{kable2001tensor, Mez04, takeda2016metaplectic, takeda2017remarks, Cai19,kaplan2022classification}). More precisely, let $M_i$ be a Levi subgroup of $G_{r_i}$, and let $M_\beta:=M_1\times\dots\times M_k$ be the resulting Levi subgroup of $G_r$ and $\omega_i$ a genuine character of $Z(\ol{G_{r_i}})$  for $i=1,\dots,k$. 
	We write
	\[M_{\beta}^{[n]}=M_{1}^{(n)}\times\dots \times M_{k}^{(n)}.\] 
	In particular, if $\beta=(r)$ with $k=1$, then we have $M_\beta^{[n]} = M_1^{(n)}$. 
	
	A genuine character $\omega$ of $Z(\ol{G_r})$ is called \emph{compatible with} $(\omega_1,\dots,\omega_k)$ if we have \[(\omega_1\rest_{Z(\ol{G_{r_1}})\cap \ol{G_{r_1}^{(n)}}}\boxtimes\dots\boxtimes\omega_k\rest_{Z(\ol{G_{r_k}})\cap \ol{G_{r_k}^{(n)}}})\rest_{Z(\ol{G_{r}})\cap \ol{G_{r}^{(n)}}}=\omega\rest_{Z(\ol{G_{r}})\cap \ol{G_{r}^{(n)}}},\]
	noting that we have  \[\mbf{p}(Z(\ol{G_{r}})\cap \ol{G_{r}^{(n)}})\subset \mbf{p}(Z(\ol{G_{r_1}})\cap \ol{G_{r_1}^{(n)}})\times\dots\times \mbf{p}(Z(\ol{G_{r_k}})\cap \ol{G_{r_k}^{(n)}})\]
	and the right-hand side is block compatible. Let $\Rep_{\omega}^{\rm fl}(\ol{M})$ (resp. $\Rep_{\omega_i}^{\rm fl}(\ol{M_{i}})$) be the category of finite length representations with each irreducible subquotient having central character $\omega$ (resp. $\omega_i$). One has a multifunctor
	\[(\boxtimes_{i=1}^{k})_{\omega}:\Rep_{\omega_1}^{\rm fl}(\ol{M_{1}})\times\dots\times \Rep_{\omega_k}^{\rm fl}(\ol{M_{k}})\rightarrow \Rep_{\omega}^{\rm fl}(\ol{M}),\ (\rho_{1},\dots,\rho_k)\mapsto(\rho_1\boxtimes\dots\boxtimes\rho_k)_\omega\] 
	called the metaplectic tensor product (MTP). Note that in \cite{kaplan2022classification}, a more general category of locally $\omega_i$ (resp. locally $\omega$) representations are considered. But for us, finite length representations will be enough. 
	
	Let $\Irr_{\omega}(\ol{M})$ (resp. $\Irr_{\omega_i}(\ol{M_{i}})$) be the set of equivalence classes of irreducible representations having central character $\omega$ (resp. $\omega_i$). The above multifunctor induces a bijection
	\[(\boxtimes_{i=1}^{k})_{\omega}:\Irr_{\omega_1}(\ol{M_{1}})\times\dots\times \Irr_{\omega_k}(\ol{M_{k}})\rightarrow \Irr_{\omega}(\ol{M}),\ (\rho_{1},\dots,\rho_k)\mapsto(\rho_1\boxtimes\dots\boxtimes\rho_k)_\omega\]
	The explicit definition of the MTP, albeit general and precise, is sometimes considered as clumsy in dealing with certain restriction and induction functors. To facilitate the definition of an extended version of MTP, we recall the following result.
	
	\begin{prop}[\cite{kaplan2022classification}*{Section 4}, \cite{zou2022metaplectic}*{Theorem 2.1}] \label{P:mot-ex}
		Let $\ol{G_r}$ be a KP cover, and we retain the above notation.
		\
		\begin{enumerate}
			\item Given $\rho_{i}\in\Irr_{\omega_i}(\ol{M_i})$, we have
			\begin{equation}\label{eqMTPext}
				\Ind_{\ol{M_{\beta}^{[n]}}}^{\ol{M}}(\rho_1\rest_{\ol{M_1^{(n)}}}\boxtimes\dots\boxtimes\rho_k\rest_{\ol{M_k^{(n)}}})\simeq m_{(r_1,\dots,r_k)}\cdot \bigoplus_{\omega}(\rho_1\boxtimes\dots\boxtimes\rho_k)_\omega, 
			\end{equation}
			where $m_{(r_1,\dots,r_k)}$ is an explicit positive integer depending only on $\ol{G_r}$ and $r_1,\dots, r_k$ but not the representations involved, and $\omega$ in the direct sum ranges over all the characters of $Z(\ol{G_r})$ compatible with $(\omega_1,\dots,\omega_k)$.
			
			\item The group of characters $X^{*}(F^{\times}/F^{\times n})$, when viewed as characters of $M/M^{(n)}$, acts transtively on the set of MTPs occurring in the right-hand side of \eqref{eqMTPext}. 
			
			\item For each $i$, let $\chi_i\in X^{*}(F^{\times}/F^{\times n})$ being regarded as a character of of $M_i/M_i^{(n)}$, then $(\rho_1\chi_1\boxtimes\dots\boxtimes\rho_k\chi_k)_\omega\simeq (\rho_1\boxtimes\dots\boxtimes\rho_k)_\omega$.
			
		\end{enumerate}
		
	\end{prop}
	
	\begin{rmk}
		
		To uniform the notation, when $\ol{G_r}$ is either a KP-cover or an S-cover, we write 
		\[(\boxtimes_{i=1}^{k})_{\omega}:\ (\rho_{1},\dots,\rho_k)\mapsto(\rho_1\boxtimes\dots\boxtimes\rho_k)_\omega\] 
		both for the MTP above (for KP-covers) and the standard tensor product (for S-covers), where in the KP-cover case we assume $\rho_i\in \Rep_{\omega_i}^{\rm fl}(\ol{M_{i}})$  with $\omega$ compatible with $(\omega_1,\dots,\omega_k)$ and in the S-cover case we assume $\rho_i\in \Rep(\ol{M_{i}})$. 
		Here in the S-cover case the presence of $\omega$ is only symbolic and does not embody any mathematical content.
	\end{rmk}
	
	\subsection{Extended metaplectic tensor product}
	
	Motivated by the Kazhdan--Patterson case, especially Proposition \ref{P:mot-ex}, we consider the following extended metaplectic tensor product (EMTP for short), which works for general BD-covers. 
	By definition, the EMTP is just the map
	\begin{equation*}
		\begin{aligned}\ol{\boxtimes_{i=1}^{k}}:&\Rep(\ol{M_{1}})\times\dots\times \Rep(\ol{M_{k}})\rightarrow \Rep(\ol{M}),\ \\
			&(\rho_{1},\dots,\rho_k)\mapsto\rho_1\ol{\boxtimes}\dots\ol{\boxtimes}\rho_k:=\Ind_{\ol{M_{\beta}^{[n]}}}^{\ol{M}}(\rho_1\rest_{\ol{M_1^{(n)}}}\boxtimes\dots\boxtimes\rho_k\rest_{\ol{M_k^{(n)}}}).
		\end{aligned}
	\end{equation*} 
	When $\rho_1,\dots,\rho_k$ are irreducible, the right-hand side is a direct sum of finitely many irreducible representations, which in particular contains the original MTP for S-cover and KP-covers (\emph{cf.} Corollary \ref{cor MTP} later on). When $k=1$, the definition still makes sense, and it becomes
	\[\Ind_{\ol{M^{(n)}}}^{\ol{M}}(\rho\rest_{\ol{M^{(n)}}}).\]
	Later we will also need to work in a more general setting. Thus, let $H_i$ be a closed subgroup of $M_i$ for each $i$ such that $H_i/H_i^{(n)}\simeq M_i/M_i^{(n)}\simeq F^{\times}/F^{\times n}$. For $r=r_1+\dots+r_k$ and $H_i$ a subgroup of $G_{r_i}$, write $H:=H_1\times\dots\times H_k$ and $H_{\beta}^{[n]}:=H_1^{(n)}\times\dots\times H_k^{(n)}$. We define similarly the EMTP by
	\begin{equation*}
		\begin{aligned}\ol{\boxtimes_{i=1}^{k}}:&\Rep(\ol{H_{1}})\times\dots\times \Rep(\ol{H_{k}})\rightarrow \Rep(\ol{H}),\ \\
			&(\rho_{1},\dots,\rho_k)\mapsto\rho_1\ol{\boxtimes}\dots\ol{\boxtimes}\rho_k:=\Ind_{\ol{H_{\beta}^{[n]}}}^{\ol{H}}(\rho_1\rest_{\ol{H_1^{(n)}}}\boxtimes\dots\boxtimes\rho_k\rest_{\ol{H_k^{(n)}}}).
		\end{aligned}
	\end{equation*} 
	The following lemma is a special case of the result in \cite{gelbart1982indistinguishability}*{Section 2}, when we take $G, H$ (in the notation there) to be $\ol{M}$ and $\ol{M_\beta^{(n)}}$ respectively.
	
	\begin{lm}\label{lem MTP corr}
		\begin{enumerate}
			
			\item Given an irreducible representation $\pi$ of $\ol{M}$, its restriction to $\ol{M_{\beta}^{[n]}}$ decomposes into a finite sum of irreducible representations. These irreducible representations of $\ol{M_{\beta}^{[n]}}$ are $M$-conjugate to each other.
			
			\item Given an irreducible representation $\tau$ of $\ol{M_{\beta}^{[n]}}$, its induction to $\ol{M}$ decomposes into a finite sum of irreducible representations. These representations are equivalent to each other up to twisting by a character of $M/M_{\beta}^{[n]}$.
			
			\item The restriction and induction from (1) and (2) induce a bijection between $M$-conjugacy classes of irreducible representations $\tau$ of $\ol{M_{\beta}^{[n]}}$ and $X^{*}(M/M_{\beta}^{[n]})$-classes of irreducible representations $\pi$ of $\ol{M}$.

		\end{enumerate}
		
	\end{lm}
	
	\begin{cor}\label{cor MTP}
		Let $\rho_{i}\in\Irr_{\epsilon}(\ol{M_i})$ for $i=1,\dots,k$. 
		
		\begin{enumerate}
			\item The representation $\rho_1\rest_{\ol{M_1^{(n)}}}\boxtimes\dots\boxtimes\rho_k\rest_{\ol{M_k^{(n)}}}$ is a direct sum of irreducible representations of the form $\tau:=\tau_1\boxtimes\dots\boxtimes\tau_k$ of the same multiplicity, where each $\tau_i$ is an irreducible constituent of $\rho_i\rest_{\ol{M_i^{(n)}}}$, $i=1,\dots,k$. 
			
			\item Let $\pi$ be an irreducible representation of $\ol{M}$ whose restriction to $\ol{M_{\beta}^{[n]}}$ intersects $\rho_1\rest_{\ol{M_1^{(n)}}}\boxtimes\dots\boxtimes\rho_k\rest_{\ol{M_k^{(n)}}}$. Then the EMTP $\rho_1\ol{\boxtimes}\dots\ol{\boxtimes}\rho_k$ is a finite direct sum of representations of the form $\pi\cdot\chi$, where $\chi\in X^{*}(M/M_{\beta}^{[n]})$. 
			
		\end{enumerate}
		
	\end{cor}
	
	\begin{proof}
		
		Statement (1) is clear. Also we remark that these $\tau$ are $M$-conjugate to each other, since for each $i$ these $\tau_i$ are $M_i$-conjugate to each other and $M_{\beta}^{[n]}$ is block compatible. Using (1), we see that $\rho_1\ol{\boxtimes}\dots\ol{\boxtimes}\rho_k$ is a direct sum of the induction $\Ind_{\ol{M_\beta^{[n]}}}^{\ol{M}}(\tau)$. Using
		Lemma \ref{lem MTP corr}, the irreducible constituents occurring in $\rho_1\ol{\boxtimes}\dots\ol{\boxtimes}\rho_k$ are all of the form $\pi\cdot\chi$, which proves (2). 
		
	\end{proof}
	
	\begin{rmk}
		
		One disadvantage of the EMTP construction is that as a multifunctor it is not associative, meaning that for $\rho_i\in\Irr_{\epsilon}(\ol{M_i})$, $i=1,2,3$, in general we may have
		\[(\rho_1\ol{\boxtimes}\rho_2)\ol{\boxtimes}\rho_3\not\simeq\rho_1\ol{\boxtimes}\rho_2\ol{\boxtimes}\rho_3\not\simeq\rho_1\ol{\boxtimes}(\rho_2\ol{\boxtimes}\rho_3).\]
	\end{rmk}
	
	Still, the EMTP construction possesses the property  of ``associativity up to multiplicity" for finite length representations, as follows.
	
	\begin{prop}\label{prop MTPIterate}
		Let $\rho_i\in\Rep_{\epsilon}(\ol{M_i})$ for $i=1,\dots,k$.
		\begin{enumerate}
			\item Let $T(\rho_1,\dots,\rho_k)$ be an ``iterated'' EMTP of $\rho_1,\dots,\rho_k$ defined by adding certain brackets on the original tensor $\rho_1\ol{\boxtimes}\dots\ol{\boxtimes}\rho_k$. Then for every $\chi_i\in X^*(M_i/M_i^{(n)})$ and thus $\chi:=\chi_1\boxtimes\dots\boxtimes\chi_k \in X^{*}(M/M_{\beta}^{[n]})$, we have $T(\rho_1,\dots,\rho_k) \cdot \chi \simeq T(\rho_1,\dots,\rho_k)$.
			
			\item Assume furthermore that each $\rho_i$ is of finite length. Then for two different iterated EMTP $T_1(\rho_1,\dots,\rho_k)$ and $T_2(\rho_1,\dots,\rho_k)$, there exist positive integers $n_1,n_2$ such that $n_1\cdot T_1(\rho_1,\dots,\rho_k)\simeq n_2\cdot T_2(\rho_1,\dots,\rho_k)$. More precisely, let $b_i$ be the number of brackets in $T_i$ for $i=1,2$, we may choose $n_1=[F^{\times}:F^{\times n}]^{b_2}$ and $n_2=[F^{\times}:F^{\times n}]^{b_1}$. 
			
		\end{enumerate}
		
	\end{prop}
	
	\begin{proof}
		
		Statement (1) follows from the definition of the EMTP. More precisely, given $\pi_i\in\Rep_{\epsilon}(\ol{M_{i}})$ and $\chi_i\in X^*(M_i)$ for each $i=1,\dots,k$ and $\chi=\chi_1\boxtimes\dots\boxtimes\chi_k\in X^{*}(M)$, we have\[\chi\cdot \Ind_{\ol{M_{\beta}^{[n]}}}^{\ol{M}}(\pi_1\rest_{\ol{M_1^{(n)}}}\boxtimes\dots\boxtimes\pi_k\rest_{\ol{M_k^{(n)}}})\simeq \Ind_{\ol{M_{\beta}^{[n]}}}^{\ol{M}}(\chi\rest_{\ol{M_{\beta}^{[n]}}}\cdot(\pi_1\rest_{\ol{M_1^{(n)}}}\boxtimes\dots\boxtimes\pi_k\rest_{\ol{M_k^{(n)}}})).\]
		Furthermore, if we know that $(\chi_i\cdot\pi_{i})\rest_{\ol{M_{i}^{(n)}}}\simeq\pi_{i}\rest_{\ol{M_{i}^{(n)}}}$ for each $i$, we get\[\chi\cdot \Ind_{\ol{M_{\beta}^{[n]}}}^{\ol{M}}(\pi_1\rest_{\ol{M_1^{(n)}}}\boxtimes\dots\boxtimes\pi_k\rest_{\ol{M_k^{(n)}}})\simeq \Ind_{\ol{M_{\beta}^{[n]}}}^{\ol{M}}(\pi_1\rest_{\ol{M_1^{(n)}}}\boxtimes\dots\boxtimes\pi_k\rest_{\ol{M_k^{(n)}}}).\]Iterating this property, we get  $T(\rho_1,\dots,\rho_k)\chi\simeq T(\rho_1,\dots,\rho_k)$ for any $\chi=\chi_1\boxtimes\dots\boxtimes\chi_k \in X^{*}(M/M_{\beta}^{[n]})$.
		
		For Statement (2), we first consider adding one bracket to $\rho_1\ol{\boxtimes}\dots\ol{\boxtimes}\rho_k$. For $1 \lest i < j \lest k$, define $\beta'=(r_i,r_{i+1},\dots,r_j)$ and
		\[M^j_i=M_i\times\dots \times M_j,\quad(M^j_i)^{[n]}_{\beta'}=M^{(n)}_i\times\dots\times M^{(n)}_{j}.\]
		We clearly have the inclusions $(M^j_i)^{[n]}_{\beta'} \subseteq (M^j_i)^{(n)} \subseteq M_i^j$ of groups. Also define
		\[\pi_{i}^j=(\text{Ind}^{\ol{M^j_i}}_{\ol{(M^j_i)_{\beta'}^{[n]}}}\rho_i|_{\ol{M^{(n)}_i}}\boxtimes\dots \boxtimes\rho_j|_{\ol{M^{(n)}_j}})|_{\ol{(M^j_i)^{(n)}}},\quad\tau_i^j=\text{Ind}^{\ol{(M^j_i)^{(n)}}}_{\ol{(M^j_i)^{[n]}_{\beta'}}}(\rho_i|_{\ol{M^{(n)}_i}}\boxtimes\dots \boxtimes\rho_j|_{\ol{M^{(n)}_j}}).\] 
		Using the Mackey formula, we have
		\begin{equation}\label{eq piijtauij}
			\pi_i^j\simeq\bigoplus_{m\in M^j_i/(M^j_i)^{(n)}}(\tau_i^j)^m.
		\end{equation}
		By definition we have
		\begin{equation}\label{eq piij}
			\begin{aligned}
				\ &\rho_1\ol{\boxtimes}\dots\ol{\boxtimes}(\rho_i\ol{\boxtimes}\dots\ol{\boxtimes}\rho_j)\ol{\boxtimes}\dots\ol{\boxtimes}\rho_k\\
				\simeq &\text{Ind}^{\ol{M}}_{\ol{M^{(n)}_1\times\dots \times (M^j_i)^{(n)}\times\dots\times M^{(n)}_k}}(\rho_1|_{\ol{M^{(n)}_1}}\boxtimes\dots \boxtimes\pi_{i}^j\boxtimes \dots \boxtimes \rho_k|_{\ol{M^{(n)}_k}}).
			\end{aligned}
		\end{equation}
		Since $\ol{M^j_i}$ commutes with $\ol{M^{(n)}_l}$ for $l<i$ or $l>j$ as subgroups of $\ol{M}$, we have 
		\begin{equation*}\begin{aligned} \ &\text{Ind}^{\ol{M^{(n)}_1\times\dots \times (M^j_i)^{(n)}\times\dots\times M^{(n)}_k}}_{\ol{M^{(n)}_1\times\dots \times (M^j_i)_{\beta'}^{[n]}\times\dots\times M^{(n)}_{k}}}(\rho_1|_{\ol{M^{(n)}_1}}\boxtimes\dots \boxtimes\rho_k|_{\ol{M^{(n)}_k}})\\\simeq &\rho_1|_{\ol{M^{(n)}_1}}\boxtimes\dots \boxtimes (\text{Ind}^{\ol{(M^j_i)^{(n)}}}_{\ol{(M^j_i)^{[n]}_{\beta'}}}(\rho_i|_{\ol{M^{(n)}_i}}\boxtimes\dots \boxtimes\rho_j|_{\ol{M^{(n)}_j}}))\boxtimes \dots \boxtimes \rho_k|_{\ol{M^{(n)}_k}}
			\end{aligned} 
		\end{equation*}
		and thus
		\begin{equation}\label{eq tauij}
			\begin{aligned}
				\ &\rho_1\ol{\boxtimes}\dots\ol{\boxtimes}\rho_k
				\simeq\text{Ind}^{\ol{M}}_{\ol{M^{(n)}_1\times\dots \times (M^j_i)^{(n)}\times\dots\times M^{(n)}_k}}(\rho_1|_{\ol{M^{(n)}_1}}\boxtimes\dots \boxtimes\tau_{i}^j\boxtimes \dots \boxtimes \rho_k|_{\ol{M^{(n)}_k}}).
			\end{aligned}
		\end{equation}
		Also, we may plug \eqref{eq piijtauij} into \eqref{eq piij} and ``pull" the conjugation by $m\in M^j_i/(M^j_i)^{(n)}$ out of the tensor product for the same reason. Comparing this with \eqref{eq tauij}, we get
		\[\rho_1\ol{\boxtimes}\dots\ol{\boxtimes}(\rho_i\ol{\boxtimes}\dots\ol{\boxtimes}\rho_j)\ol{\boxtimes}\dots\ol{\boxtimes}\rho_k\simeq |M^j_i/(M^j_i)^{(n)}|\cdot \rho_1\ol{\boxtimes}\dots\ol{\boxtimes}\rho_i\ol{\boxtimes}\dots\ol{\boxtimes}\rho_j\ol{\boxtimes}\dots\ol{\boxtimes}\rho_k,\]
		where $|M^j_i/(M^j_i)^{(n)}|=[F^{\times}:F^{\times n}]$.
		In general, Statement (2) follows from an iterated application of the above special case.
		
	\end{proof}

	Finally, we need the following two lemmas describing the EMTP under the natural action of permutation matrix and taking contragredient.
	
	\begin{lm}\label{lem WeylonMTP}
		
		Let $M=M_1\times\dots\times M_k$ with $M_i$ being a Levi subgroup of $G_{r_i}$ for each $i$ and $r_1+\dots+r_k=r$, let $\varsigma\in\mfr{S}_k$ be a permutation. Let $M'=M_{\varsigma(1)}\times\dots\times M_{\varsigma(k)}$ and $w$ be a permutation matrix in $G_r$ such that $M'=M^w$. Then for $\rho_i\in\Rep_{\epsilon}(M_i)$, we have
		\[(\rho_1\ol{\boxtimes}\dots\ol{\boxtimes}\rho_k)^w\simeq \rho_{\varsigma(1)}\ol{\boxtimes}\dots\ol{\boxtimes}\rho_{\varsigma(k)}.\]	
	\end{lm}
	
	\begin{proof}
		Let $\beta=(r_1,\dots,r_k)$ and $\beta'=(r_{\varsigma(1)},\dots,r_{\varsigma(k)})$ be two compositions of $r$. Let
		$M_{\beta}^{[n]} =M_1^{(n)}\times\dots\times M_k^{(n)}$ and $M_{\beta'}^{[n]} =M_{\varsigma(1)}^{(n)}\times\dots\times M_{\varsigma(k)}^{(n)}$ be the two subgroups of $M$ and $M'$ respectively. We have $M_{\beta'}^{[n]}=(M_{\beta}^{[n]})^w$ and moreover
		\begin{equation*}
			\begin{aligned}
				(\rho_1\ol{\boxtimes}\dots\ol{\boxtimes}\rho_k)^w&\simeq(\Ind_{\ol {M_{\beta}^{[n]}}}^{\ol M}(\rho_1\rest_{\ol {M_1^{(n)}}}\boxtimes\dots\boxtimes\rho_k\rest_{\ol {M_k^{(n)}}}))^w	\\
				& \simeq\Ind_{\ol{M_{\beta'}^{[n]}}}^{\ol{M'}}((\rho_1\rest_{\ol{M_1^{(n)}}}\boxtimes\dots\boxtimes\rho_k\rest_{\ol{M_k^{(n)}}})^w)\\
				& \simeq\Ind_{\ol{M_{\beta'}^{[n]}}}^{\ol{M'}}(\rho_{\varsigma(1)}\rest_{\ol{M_{\varsigma(1)}^{(n)}}}\boxtimes\dots\boxtimes\rho_{\varsigma(k)}\rest_{\ol{M_{\varsigma(k)}^{(n)}}})\\
				& \simeq\rho_{\varsigma(1)}\ol{\boxtimes}\dots\ol{\boxtimes}\rho_{\varsigma(k)}.
			\end{aligned}
		\end{equation*}
		Here, we used the block-commutativity of the $\ol{M_j^{(n)}}$'s in the proof.
		
	\end{proof}
	
	\begin{lm}\label{lem MTPcontragredient}
		
		Under the same assumption of Lemma \ref{lem WeylonMTP}, we have
		\[(\rho_1\ol{\boxtimes}\dots\ol{\boxtimes}\rho_k)^{\vee}\simeq (\rho_1^{\vee}\ol{\boxtimes}\dots\ol{\boxtimes}\rho_k^{\vee}),\]
		where $\rho_i^{\vee}\in\Rep_{\epsilon^{-1}}(\ol{M_i})$ for each $i$ and $(\rho_1\ol{\boxtimes}\dots\ol{\boxtimes}\rho_k)^{\vee}\in\Rep_{\epsilon^{-1}}(\ol{M})$.
		
	\end{lm}
	
	The proof follows from the compatibility of contragredient with restriction and induction functors.
	
	\subsection{Whittaker and semi-Whittaker models}
	
	We briefly discuss the Whittaker model and semi-Whittaker model for a BD-cover $\ol{G_r}$ and its Levi subgroup $\ol{M}$. We fix a non-trivial additive character $\psi_F: F \to \C^\times$.
	
	Let $U_r \subset G_r$ be the unipotent group of upper triangular matrices and let 
	$$\psi_r:U_r\rightarrow\C^{\times},\ (v_{ij})_{1\lest i,j \lest r}\mapsto \psi_F\Big( \sum_{i=1}^{r-1}v_{i,i+1}\Big)$$
	be the standard generic character of $U_r$ associated with $\psi_F$. Such a pair $(U_r,\psi_r)$, with $U_r$ viewed as a subgroup of $\ol{G_r}$ via its canonical splitting, is called the standard Whittaker pair of $\ol{G_r}$. 
	
	Furthermore, we define a Whittaker pair of $\ol{M}$ as follows. If $\ol{M}$ is block-diagonal and thus a standard Levi subgroup, then $(M\cap U_r,\psi_r\rest_{M\cap U_r})$ is called the standard Whittaker pair of $\ol{M}$. In general, a Whittaker pair of $\ol{M}$ is a $G$-conjugation to the standard Whittaker pair of some block-diagonal Levi subgroup.

	Given a Whittaker pair $(U_M,\psi_M)$ of $\ol{M}$ and a smooth representation $\pi_M$ of $\ol{M}$, we define the associated Whittaker space
	\[\Wh(\pi_M)=\Hom_{U_M}(\pi_M\rest_{U_M},\psi_{M}).\]
	It is well-known that the functor 
	$$\Wh:\Rep(\ol{M})\rightarrow \{\C\ \text{vector spaces}\}$$ is exact.
	Moreover, we know that $\dim \Wh(\pi_M) < \infty$ for every irreducible $\pi_M$, see \cite{patel2014theorem}*{Theorem 2}. However, unlike in the linear reductive group case, such dimension is in general not bounded above by one for genuine irreducible representations, i.e., multiplicity-one property fails for covers.
	
	We discuss the semi-Whittaker model of representations of $\ol{G_r}$ associated with a composition $\lambda=(r_1,\dots,r_k)$ of $r$, following the terminology in \cite{Cai19}. Define some notation here: \begin{itemize}
		\item $M_{\lambda}=G_{r_1}\times\dots\times G_{r_k}$ is the block-diagonal standard Levi subgroup of $G_r$,
		\item $U_{\lambda}$ is the standard unipotent radical of the block upper triangular parabolic subgroup of $G_r$ with Levi factor $M_{\lambda}$,
		\item $(U_r\cap M_{\lambda},\psi_r\rest_{U_r\cap M_{\lambda}})$ is a generic Whittaker pair of $\ol{M_{\lambda}}$,
		\item $J_{U_{\lambda}}(\pi)$ denotes the Jacquet module of $\pi$ as a smooth representation of $\ol{M_{\lambda}}$.
	\end{itemize}Consider the character $\psi_{\lambda}$ which is trivial on $U_{\lambda}$ and is the standard generic character on $U_r \cap M_{\lambda}$:
	\[\psi_{\lambda}:U_r\rightarrow\C^{\times},\ (v_{ij})_{1\lest i,j\lest r}\mapsto\psi_F\Big(\sum_{i=1}^k\sum_{j=1}^{r_i-1}v_{\sum_{s=1}^{i-1}r_s+j,\sum_{s=1}^{i-1}r_s+j+1}\Big).\] Then $(U_r,\psi_{\lambda})$ is called the standard semi-Whittaker pair of $\ol{G_r}$ associated with $\lambda$. 
	It is known that the set of  $G_r$-conjugacy classes of semi-Whittaker pairs are in bijection with the set of partitions of $r$, or equivalently, the set of nilpotent orbits of $G_r$. In any case, we define the $\lambda$-semi-Whittaker space of a smooth representation $\pi$ of $\ol{G}$ to be
	$$\Wh_{\lambda}(\pi):=\Hom_{U_r}(\pi\rest_{U_r},\psi_{\lambda}).$$
	Using the Frobenius reciprocity, we have
	\begin{equation}\label{eq generalizedWhit}
		\Wh_{\lambda}(\pi) \simeq\Hom_{U_r\cap M_{\lambda}}(J_{U_{\lambda}}(\pi),\psi\rest_{U_r\cap M_{\lambda}}).
	\end{equation}
	Thus, we have 
	$$\Wh_{\lambda}(\pi) \simeq \Wh(J_{U_{\lambda}}(\pi))$$
	as vector spaces, which is finite dimensional if $\pi$ is of finite length. 
	
	Finally, we discuss the Whittaker model of an EMTP of irreducible representations. Let $r=r_1+\dots+r_k$. For $\rho_i\in\Irr_{\epsilon}(\ol{G_{r_i}})$, taking the EMTP and using the Mackey formula, we have
	\[\dim\Wh((\rho_1\ol \boxtimes\dots\ol \boxtimes\rho_k))=\prod_{i=1}^{k} \dim(\Wh(\Ind_{\ol {G_{r_{i}}^{(n)}}}^{\ol{G_{r_i}}}\rho_i))\]
	Let $m$ (resp. $m_i$) be the total multiplicity of $\rho_1\ol \boxtimes\dots\ol \boxtimes\rho_k$ (resp. $\Ind_{\ol{G_{r_{i}}^{(n)}}}^{\ol{G_{r_i}}}\rho_i$) when written as a direct sum of irreducible representations. Using the fact that those irreducible summands have the same Whittaker dimension (since they only differ by a character twist by Corollary \ref{cor MTP}), we have
	\begin{equation}\label{eq WhitdimBD}
		m\cdot \dim\Wh(\rho)=\prod_{i=1}^{k}m_i\cdot\dim(\Wh(\rho_i)),
	\end{equation}
	where $\rho$ is any irreducible summand of $\rho_1\ol \boxtimes\dots\ol \boxtimes\rho_k$.
	
	\section{Bernstein--Zelevinsky theory}
	In this section, we discuss the Bernstein--Zelevinsky theory for a BD-cover $\ol{G_r}$, following \cite{bernstein1977induced}. We play emphasis on elucidating those non-trivial aspects of the theory when working for covers.
	
	\subsection{Extended Bernstein-Zelevinsky product}
	
	Let $\ol{G}$ be an $n$-fold cover of a reductive group $G$ over $F$. Consider a group $P=MN$ with $M,N$ being subgroups of $G$, $M\cap N=\{1\}$ and $M$ normalizing $N$. We assume $N$ is a unipotent subgroup of $G$, identified as a subgroup of $\ol{G}$, then $\ol{P}=\ol{M}N$ is a subgroup of $\ol{G}$ satisfying the above conditions as well. Let $\psi$ be a character of $N$ normalized by $M$, then as in \cite{bernstein1977induced}*{\S 1.8}, we define the normalized non-compact and compact induction functors and restriction functor
	\[I_{N,\psi},i_{N,\psi}:\Rep(\ol{M})\rightarrow\Rep(\ol{G}),\quad J_{N,\psi}:\Rep(\ol{G})\rightarrow\Rep(\ol{M}).\]
	In particular, in the case $P$ is a parabolic subgroup of $G$ with $N$ its unipotent radical and $\psi=\mbm{1}$, we simply write 
	$$i_{\ol P}^{\ol G}=i_{N, \mbm{1}} \text{ and } J_{N}=J_{N,\mbm{1}}$$
	for the related parabolic induction and Jacquet module functor.
	
	From now on we focus on the case $G=G_r$. Let $\ol{G_r}$ be a BD-cover of $G_r$. For $r=r_1+\dots+r_k$ and $\beta=(r_1,\dots,r_k)$, we write $M=G_{r_1}\times\dots\times G_{r_k}$ and let $P=MN$ be the standard parabolic subgroup with Levi factor $M$ and unipotent radical $N$. We define the extended Bernstein--Zelevinsky product 
	\begin{equation*}
		\begin{aligned}
			\ol{\times_{i=1}^k}:=i_{\ol{P}}^{\ol{G_r}} \circ \ol{\boxtimes_{i=1}^k}: &\Rep(\ol{G_{r_1}})\times\dots\times \Rep(\ol{G_{r_k}})\rightarrow \Rep(\ol{G_r}),\ \\
			&(\rho_{1},\dots,\rho_k)\mapsto\rho_1\ol{\times}\dots\ol{\times}\rho_k:=i_{\ol{P}}^{\ol{G_r}}(\ol{\boxtimes}_i \rho_i).
		\end{aligned}
	\end{equation*}
	It maps finite length representations to finite length representations.

	We emphasize that an associative law up to multiplicity is satisfied.
	
	\begin{prop}\label{prop BZPasso}
		
		For $r=r_1+r_2+r_3$ and $\rho_i\in\Rep_{\epsilon}^{\rm fl}(\ol{G_{r_i}})$,
		\[[F^{\times}:F^{\times n}]\cdot (\rho_1\ol{\times}\rho_2\ol{\times}\rho_3)\simeq (\rho_1\ol{\times}\rho_2)\ol{\times}\rho_3\simeq \rho_1\ol{\times}(\rho_2\ol{\times}\rho_3).\]
		
	\end{prop}
	
	\begin{proof}
		
		First, using Proposition \ref{prop MTPIterate}.(2) we have 
		\[[F^{\times}:F^{\times n}]\cdot (\rho_1\ol{\boxtimes}\rho_2\ol{\boxtimes}\rho_3)\simeq (\rho_1\ol{\boxtimes}\rho_2)\ol{\boxtimes}\rho_3\simeq \rho_1\ol{\boxtimes}(\rho_2\ol{\boxtimes}\rho_3).\]
		Let $P=P_{r_1,r_2,r_3}$ be the standard parabolic subgroup whose Levi subgroup is $M_{r_1,r_2,r_3}\simeq G_{r_1}\times G_{r_2}\times G_{r_3}$. Taking the parabolic induction $i_{\ol{P}}^{\ol{G}}$,
		we need to show that
		\[i_{\ol{P}}^{\ol{G}}((\rho_1\ol{\boxtimes}\rho_2)\ol{\boxtimes}\rho_3)\simeq (\rho_1\ol{\times}\rho_2)\ol{\times}\rho_3\quad\text{and}\quad i_{\ol{P}}^{\ol{G}}(\rho_1\ol{\boxtimes}(\rho_2\ol{\boxtimes}\rho_3))\simeq \rho_1\ol{\times}(\rho_2\ol{\times}\rho_3).\]
		This follows from a direct calculation using associativity of induction functor and the geometric lemma. 
		
	\end{proof}
	
	We specialize to the case where $\ol{G_r}$ is either a KP-cover or an S-cover and each $\rho_i$ 
	is genuine and of finite length. In the KP-cover case, we further assume that for each $i$ every irreducible subquotient of $\rho_i$ has genuine central character $\omega_i$ of $Z(\ol{G_{r_i}})$. Moreover, we assume that $(\omega_1,\dots,\omega_k)$ is a sequence of characters compatible with a genuine character $\omega$ of $Z(\ol{G_r})$. Then we define the usual Bernstein--Zelevinsky product 
	\begin{equation*}
		\begin{aligned}(\rho_1\times\dots\times\rho_k)_{\omega}:=i_{\ol{P}}^{\ol{G_r}}((\rho_1\boxtimes\dots\boxtimes\rho_k)_{\omega})
		\end{aligned}
	\end{equation*}
	as in \cite{kaplan2022classification}*{\S 5.2, Remark 7.6}.
	Note that in the S-cover case, the representation $(\rho_1\times\dots\times\rho_k)_{\omega}$ is the usual Bernstein--Zelevinsky product and the notation $\omega$ is only symbolic. 
	
	The proposition below follows from the above discussion and Corollary \ref{cor MTP}.
	
	\begin{prop}\label{prop irredBZPandEBZP}
		
		If $\ol{G_r}$ is either a KP-cover or an S-cover and each $\rho_i$ is irreducible, then the usual Bernstein--Zelevinsky product $(\rho_1\times\dots\times\rho_k)_{\omega}$ is a direct summand of the extended Bernstein--Zelevinsky product $\rho_1\ol\times\dots\ol\times\rho_k$. 	
	\end{prop}
	
	\subsection{Cuspidal support and $n$-equivalence classes}
	For any irreducible representation $\pi$ of $\ol{G}$, there exist a Levi subgroup $M$ of $G$ and a cuspidal representation $\rho$ of $\ol{M}$ such that $\pi$ is a subquotient of the parabolic induction $i_{\ol{P}}^{\ol{G}}(\rho)$, where $P$ is a parabolic subgroup of $G$ with Levi subgroup $M$. Moreover, the pair $(\ol{M},\rho)$ is unique up to $G$-conjugacy, and is called the \emph{cuspidal support} of $\pi$.
	
	Write $M=G_{r_1}\times\dots\times G_{r_k}$. Given a genuine irreducible representation $\rho$ of $\ol{M}$, using Lemma \ref{lem MTP corr} and Corollary \ref{cor MTP}, there exist genuine irreducible representations $\rho_i$ of each $G_{r_i}$, such that $\rho$ is an irreducible summand of $\rho_{1}\ol{\boxtimes}\dots\ol{\boxtimes}\rho_k$. By considering the matrix coefficient, the representation $\rho$ is cuspidal if and only if each $\rho_i$ is cuspidal. On the one hand, for each $i$ the $X^{*}(G_{r_i}/G_{r_i}^{(n)})$-orbit of $\rho_i$ is uniquely determined by $\rho$. On the other hand, replacing $\rho_i$ by another representation in the same $X^{*}(G_{r_i}/G_{r_i}^{(n)})$-orbit does not change the EMTP. Finally, an irreducible representation $\rho'$ occurs as a summand in the above EMTP if and only if it lies in the $X^{*}(M/M_{\beta}^{[n]})$-orbit of $\rho$ with $\beta=(r_1,\dots,r_k)$. 
	
	For any irreducible representation $\rho$ of $\ol{M}$, we denote by $[\rho]$ the $X^{*}(M/M_{\beta}^{[n]})$-orbit of $\rho$ (which in particular includes the $k=1$ case), called the \emph{$n$-equivalence class} of $\rho$. If moreover $\rho$ is cuspidal, then $[\rho]$ is called an \emph{$n$-cuspidal equivalence class}. As a result, the above discussion shows that the set of genuine $n$-cuspidal equivalence classes $[\rho]$ of $\ol{M}$ corresponds bijectively to the set of sequences 
	\[([\rho_1],\dots,[\rho_k]),\quad  \rho_i\in\Cusp_{\epsilon}(\ol{G_{r_i}}),\ r_1+\dots+r_k=r.\] 
	Using Lemma \ref{lem WeylonMTP}, the set of $G$-conjugacy classes of all possible pairs $(\ol{M},[\rho])$ corresponds bijectively to the set of multisets (meaning that there is no specific order for $[\rho_1],\dots,[\rho_k]$) \[[\rho_1]+\dots+[\rho_k],\quad \rho_i\in\Cusp_{\epsilon}(\ol{G_{r_i}}),\ r_1+\dots+r_k=r.\]
	For a genuine irreducible representation $\pi$ of $\ol{G_r}$, we call the related multiset $[\rho_1]+\dots+[\rho_k]$ the \emph{$n$-cuspidal support} of $\pi$, if $\pi$ occurs as a subquotient in $\rho_1\ol{\times}\dots\ol{\times}\rho_k$.
	
	\begin{rmk}
		
		For $M=G_r$ and an irreducible representation $\rho$ of $\ol{M}$, the $n$-equivalence class $[\rho]$ is exactly the same as the so-called weak equivalence class of $\rho$ in \cite{kaplan2022classification}*{Definition 4.2} (the latter is only defined for representations of KP-covers). However, for general Levi subgroup $M$, these two concepts are different. Also, our $n$-cuspidal support is exactly the weak cuspidal support omitting the central character $\omega$ in \cite{kaplan2022classification}*{\S 5.4}. 
		
	\end{rmk}
	
	\subsection{Bernstein--Zelevinsky derivatives}
	
	Now we discuss the theory of Bernstein--Zelevinsky derivatives. Let $P_r$ be  the mirabolic subgroup of $G_r$ (i.e. elements in $G_r$ with last row being $(0,\dots,0,1)$) and $V_r \subseteq P_r$ the unipotent radical of $P_r$. Define the character $\theta_r$ of $V_r$ by $\theta_r((v_{ij}))=\psi_F(v_{r-1,r})$. Then $P_r$ normalizes $\theta_r$. Define functors
	\[\Psi^-:\Rep(\ol{P_r})\rightarrow\Rep(\ol{G_{r-1}}),\qquad \Psi^+:\Rep(\ol{G_{r-1}})\rightarrow\Rep(\ol{P_r})\]
	\[\Phi^-:\Rep(\ol{P_r})\rightarrow\Rep(\ol{P_{r-1}}),\qquad  \Phi^+:\Rep(\ol{P_{r-1}})\rightarrow\Rep(\ol{P_r})\]
	\[\Psi_n^-:\Rep(\ol{P_{r}^{(n)}})\rightarrow\Rep(\ol{G_{r-1}^{(n)}}),\qquad \Psi_n^+:\Rep(\ol{G_{r-1}^{(n)}})\rightarrow\Rep(\ol{P_r^{(n)}})\]
	\[\Phi_n^-:\Rep(\ol{P_r^{(n)}})\rightarrow\Rep(\ol{P_{r-1}^{(n)}}),\qquad  \Phi_n^+:\Rep(\ol{P_{r-1}^{(n)}})\rightarrow\Rep(\ol{P_r^{(n)}})\]
	by 
	\[\Psi^-,\Psi_n^-:=J_{V_r,\mbm{1}},\quad \Phi^-, \Phi_n^-:=J_{V_r,\theta_r},\quad \Psi^+,\Psi_n^+:=i_{V_r,\mbm{1}},\quad \Phi^+,\Phi_n^+:=i_{V_r,\theta_r},\]
	where $V_r$ is viewed as a subgroup of $\ol{G_r}$ via the canonical splitting. The definition here is identical with that in \cite{bernstein1977induced}*{\S 3.2}. Note that all the elements in $N_r$ are of determinant 1, thus the related functors $\Psi_n^-,\Psi_n^+,\Phi_n^-,\Phi_n^+$ are well-defined. All the results in \cite{bernstein1977induced}*{Section 3} work here without essential change.
	
	Let us recall the notion of Bernstein--Zelevinsky derivative. Consider $\tau\in\Rep(\ol{P_r})$ (resp. $\Rep(\ol{P_r^{(n)}})$). Define 
	$$\tau^{(k)}:=\Psi^-\circ(\Phi^{-})^{k-1}(\tau)\in\Rep(\ol{G_{r-k}}) \text{ (resp. $\Rep(\ol{G_{r-k}^{(n)}})$)},$$ called the $k$-th derivative of $\tau$. Then, $\tau$ is glued from the representations $(\Phi^+)^{k-1}\circ\Psi^+(\tau^{(k)})$ (resp. $(\Phi_n^+)^{k-1}\circ\Psi_n^+(\tau^{(k)})$), $k=1,2,\dots,r$. Also, the representation is generic if and only if $\tau^{(r)}\neq 0$ and the dimension of $\tau^{(r)}$ (as a representation of $\mu_n$) is exactly the Whittaker dimension of $\tau$. For $\pi\in \Rep(\ol{G_r})$ (resp. $\Rep(\ol{G_r^{(n)}})$), we define $\pi^{(k)}=(\pi\rest_{\ol{P_r}})^{(k)}$ (resp. $\pi^{(k)}=(\pi\rest_{\ol{P_r^{(n)}}})^{(k)}$). In this way, we get a functor
	\[(\cdot)^{(k)}:\Rep(\ol{G_r})\rightarrow\Rep(\ol{G_{r-k}}).\]
	In particular, $\pi^{(r)}\simeq \Hom_{U_r}(\pi,\psi_r)$ is the Whittaker space of $\pi$. If $\pi$ is of finite length, then $\pi^{(r)}$ is finite dimensional.
	
	We also define the ``partial version'' of the above functors. For a fixed $r_0\gest 1$ and each $r\gest 1$, consider the embedding $\iota_1:V_{r}\hookrightarrow G_{r}\times G_{r_0},\ x\mapsto(x,1)$ given by the natural injection on the first coordinate.  For any closed subgroup $H$ of $G_{r_0}$, we define
	\[\,^1\Psi^-:\Rep(\ol{P_r\times H})\rightarrow\Rep(\ol{G_{r-1}\times H}),\qquad \,^1\Psi^+:\Rep(\ol{G_{r-1}\times H})\rightarrow\Rep(\ol{P_r\times H})\]
	\[\,^1\Phi^-:\Rep(\ol{P_r\times H})\rightarrow\Rep(\ol{P_{r-1}\times H}),\qquad  \,^1\Phi^+:\Rep(\ol{P_{r-1}\times H}),\rightarrow\Rep(\ol{P_r\times H})\]
	\[\,^1\Psi_n^-:\Rep(\ol{P_r^{(n)}\times H})\rightarrow\Rep(\ol{G_{r-1}^{(n)}\times H}),\qquad \,^1\Psi_n^+:\Rep(\ol{G_{r-1}^{(n)}\times H})\rightarrow\Rep(\ol{P_r^{(n)}\times H})\]
	\[\,^1\Phi_n^-:\Rep(\ol{P_r^{(n)}\times H})\rightarrow\Rep(\ol{P_{r-1}^{(n)}\times H}),\qquad  \,^1\Phi_n^+:\Rep(\ol{P_{r-1}^{(n)}\times H})\rightarrow\Rep(\ol{P_r^{(n)}\times H})\]
	by 
	\[\,^1\Psi^-,\,^1\Psi_n^-:=J_{V_r,\mbm{1}},\ \,^1\Phi^-, \,^1\Phi_n^-:= J_{V_r,\theta_r},\ \,^1\Psi^+,\,^1\Psi_n^+:=i_{V_r,\mbm{1}},\ \,^1\Phi^+,\,^1\Phi_n^+:=i_{V_r,\theta_r}\]
	with $V_r$ being identified with a subgroup of $\ol{G_r\times G_{r_0}}$ via the composite of the embedding $\iota_1$ and the canonical splitting. Similarly, for the embedding $\iota_2:V_{r}\hookrightarrow G_{r_0}\times G_{r},\ x\mapsto(1,x)$ given by the natural injection on the second coordinate, in parallel we define functors
	\[\,^2\Psi^-,\,^2\Psi_n^-,\,^2\Phi^-, \,^2\Phi_n^-, \,^2\Psi^+,\,^2\Psi_n^+,\,^2\Phi^+,\,^2\Phi_n^+.\]
	Let $M_{r-k,k}=G_{r-k}\times G_k$ and consider $M_{r-k,k}^{[n]}=G_{r-k}^{(n)}\times G_k^{(n)}$.  Define the partial derivative functor (with respect to the second factor)
	\[\partial^{s}:\Rep(\ol{M_{r-k,k}})\rightarrow\Rep(\ol{M_{r-k,k-s}})\ (\text{resp.}\ \Rep(\ol{M_{r-k,k}^{[n]}})\rightarrow\Rep(\ol{M_{r-k,k-s}^{[n]}})\]
	by first restricting to $\ol{G_{r-k}\times P_k}$ (resp. $\ol{G_{r-k}^{(n)}\times P_k^{(n)}}$), and then composing the functor $\,^2\Psi^{-}\circ(\,^2\Phi^{-})^{s-1}$. 
	
	The relation between the derivative and the partial derivative is as follows. Let $N_{r-k,k}$ be the standard unipotent radical of $M_{r-k,k}$, then we get 
	\begin{equation} \label{E:r-BZ}
		\partial^{k}\circ J_{N_{r-k,k}}=(\cdot)^{(k)}:\Rep(\ol{G_r})\rightarrow\Rep(\ol{G_{r-k}}).
	\end{equation}

	The following result is immediate.
	\begin{lm}\label{L:cuspderivative}
		Let $\pi \in\Cusp_{\epsilon}(\ol{G_r})$. Then $\pi^{(k)}$ is zero unless $k=r$, and in this case $\pi^{(r)}$ is a finite direct sum of the character $\epsilon$ of $\mu_n$.
	\end{lm}
	
	More generally, since the Whittaker dimension of a finite length representation is finite, the map $(\cdot)^{(k)}$ sends finite length representations to finite length representations.
	\vskip 10pt
	
	Recall that for an $\ell$-group $G$, representations $(\tau,V_1),(\pi,V_2)$ and a character $\chi$  of $G$, a $\chi$-pairing of $\tau$ and $\pi$ is a bilinear form 
	$$B: V_1\times V_2 \longrightarrow \C$$
	that is $(G,\chi)$-equivariant (\cite{bernstein1977induced}*{Definition 3.6}). Such a bilinear form corresponds to an element in $\Hom_{G}(\pi,(\tau\chi)^{\vee})\simeq\Hom_{G}(\tau,(\pi\chi)^{\vee})$. The form $B$ is called \emph{non-degenerate} with respect to $\pi$ (resp. $\tau$) if the corresponding morphism $\pi\rightarrow(\tau\chi)^{\vee}$ (resp. $\tau\rightarrow(\pi\chi)^{\vee}$) is an embedding.
	
	\begin{prop}[\cite{bernstein1977induced}*{Proposition 3.8}]\label{prop derivativepairing}
		
		Let $\pi,\tau$ be representations of $\ol{P_r}$. Suppose that there exists a $\delta_{\ol{P_r}}^{-1}$-pairing 
		of $\pi$ and $\tau$ which is non-degenerate with respect to $\tau$. Let $\tau^{(k)}$ be the highest derivative of $\tau$. Then there exists a $\mbm{1}$-pairing of $\pi^{(k)}$ and $\tau^{(k)}$ which is non-degenerate with respect to $\tau^{(k)}$.
		
	\end{prop}
	
	\subsection{Calculating derivatives in a special case}\label{subsubsection derivativeBD}
	
	Now we consider a special case of $\pi \in \Irr_\epsilon(\ol{G_r})$ of a BD-cover $\ol{G_r}$, where we assume that the Jacquet module $J_{N_{r-k,k}}(\pi)$ is either 0 or irreducible for a given $k$. In the former case, we have $\pi^{(k)}=0$. In the latter case, there is an irreducible genuine representation $\tau=\tau_1\boxtimes\tau_2$ of $\ol{M_{r-k,k}^{[n]}}=\ol{G_{r-k}^{(n)}\times G_k^{(n)}}$, unique up to $M_{r-k,k}$-conjugacy, such that $J_{N_{r-k,k}}(\pi)$ is an irreducible summand of $\Ind_{\ol{M_{r-k,k}^{[n]}}}^{\ol{M_{r-k,k}}}(\tau)$. Using the geometric lemma, we have
	\begin{equation*}
		\begin{aligned}
			\partial^{k}(\Ind_{\ol{M_{r-k,k}^{[n]}}}^{\ol{M_{r-k,k}}}(\tau))&\simeq \bigoplus_{g\in M_{r-k,k}^{[n]}\backslash M_{r-k,k}/G_{r-k}\times\{1\}}\Ind_{\ol{G_{r-k}^{(n)}}}^{\ol{G_{r-k}}}(\partial^{k}((\tau_1\boxtimes \tau_2)^g))\\
			&\simeq\bigoplus_{g_2\in G_{k}/G_{k}^{(n)}}\dim(\mathrm{Wh}(\tau_2^{g_2}))\cdot\Ind_{\ol{G_{r-k}^{(n)}}}^{\ol{G_{r-k}}}(\tau_1)\\
			&\simeq \dim(\mathrm{Wh}(\Ind_{\ol{G_k^{(n)}}}^{\ol{G_k}}\tau_2)))\cdot\Ind_{\ol{G_{r-k}^{(n)}}}^{\ol{G_{r-k}}}(\tau_1)
		\end{aligned}
	\end{equation*}
	Write $m$, $m_1$, $m_2$ for the total multiplicity of $\Ind_{\ol{M_{r-k,k}^{[n]}}}^{\ol{M_{r-k,k}}}(\tau)$, $\Ind_{\ol{G_{r-k}^{(n)}}}^{\ol{G_{r-k}}}(\tau_1)$ and $\Ind_{\ol{G_{k}^{(n)}}}^{\ol{G_{k}}}(\tau_2)$ respectively, when decomposing the representation into direct sum of irreducible representations. Using Lemma \ref{lem MTP corr}, we see that the irreducible representations occurring in the direct sum are of the same Whittaker dimension. Let $\pi_1$ be an irreducible summand of $\Ind_{\ol{G_{r-k}^{(n)}}}^{\ol{G_{r-k}}}(\tau_1)$ and $\pi_2$ be that of $\Ind_{\ol{G_{k}^{(n)}}}^{\ol{G_{k}}}(\tau_2)$, we get 
	\[\partial^{k}(\Ind_{\ol{M_{r-k,k}^{[n]}}}^{\ol{M_{r-k,k}}}(\tau))=m_2\cdot \dim(\mathrm{Wh}(\pi_2))\cdot \Ind_{\ol{G_{r-k}^{(n)}}}^{\ol{G_{r-k}}}(\tau_1).\]
	On the other hand, using Lemma \ref{lem MTP corr}.(2) we have
	\[\partial^{k}(\Ind_{\ol{M_{r-k,k}^{[n]}}}^{\ol{M_{r-k,k}}}(\tau))=\bigoplus_{\chi\in X^{*}(M_{r-k,k}/M_{r-k,k}^{[n]})}m_{\chi}\cdot \partial^{k}(J_{N_{r-k,k}}(\pi)\cdot \chi),\]
	where $\sum_{\chi\in X^{*}(M_{r-k,k}/M_{r-k,k}^{[n]})}m_{\chi}=m$. Write $\chi=\chi_1\boxtimes\chi_2$, we have
	\[\partial^{k}(J_{N_{r-k,k}}(\pi)\cdot \chi)= \partial^{k}(J_{N_{r-k,k}}(\pi))\cdot\chi_1.\]
	Using the above facts, we have that
	\[\partial^{k}(J_{N_{r-k,k}}(\pi))\simeq \bigoplus_{\chi_1\in X^{*}(G_{r-k}/G_{r-k}^{(n)})}m_{\chi_1}'\cdot (\pi_1 \chi_1)\]
	for some $m_{\chi_1}'\in \N$ such that $m\cdot\sum_{\chi_1\in X^{*}(G_{r-k}/G_{r-k}^{(n)})}m_{\chi_1}'=m_1 m_2\dim(\mathrm{Wh}(\pi_2))$. Thus we have shown the following:
	
	\begin{prop}
		
		Let $\pi\in\Irr_{\epsilon}(\ol{G_r})$. Assume that the Jacquet module $J_{N_{r-k,k}}(\pi)$ is an irreducible constituent of $\pi_1\ol{\boxtimes}\pi_2$ with $\pi_1\in \Irr_{\epsilon}(\ol{G_{r-k}})$ and $\pi_2\in \Irr_{\epsilon}(\ol{G_{k}})$. Then $\pi^{(k)}$ is a direct sum of irreducible representations of the form 
		\[\pi_1\chi_1,\quad\chi_1\in X^*(G_{r-k}/G_{r-k}^{(n)}),\] 
		where the total multiplicity is $m_1 m_2\dim(\mathrm{Wh}(\pi_2))/m$.
	\end{prop}
	
	If $\ol{G_r}$ is a KP-cover, then $m_1m_2/m$ is a number depending only on groups but independent of representations. More precisely,  assuming $\gcd(p,n)=1$ for simplicity and using the formula in \cite{zou2022metaplectic}*{Proposition 2.3}, we have 
	\[m_1=m_2=n^2\quad\text{and}\quad m=n^4d_r/d_{r-k}d_k\]
	and
	\begin{equation}\label{eq tensormult}
		m_1m_2/m=d_kd_{r-k}/d_r.
	\end{equation}

	\subsection{Compatibility of the extended Bernstein--Zelevinsky product with derivatives} Following \cite{bernstein1977induced}*{Section 4}, we study the compatibility of the extended Bernstein--Zelevinsky product with derivatives. 
	
	We first define the three types of extended Bernstein--Zelevinsky products following \cite{bernstein1977induced}*{\S 4.12}. Fix $r_1+r_2=r$. We define here the extended product $\ol{\times}:\Rep_{\epsilon}(\ol{H_1})\times\Rep_{\epsilon}(\ol{H_2})\rightarrow\Rep_{\epsilon}(\ol{H})$ in the following three cases:
	\begin{enumerate}
		\item $H_1=G_{r_1}$, $H_2=G_{r_2}$ and $H=G_r$;
		\item $H_1=P_{r_2}$, $H_2=G_{r_1}$ and $H=P_r$;
		\item $H_1=G_{r_1}$, $H_2=P_{r_2}$ and $H=P_r$.
	\end{enumerate}
	We define the product $\ol{\times}$ to be the composition of the extended metaplectic tensor product 
	\[\ol{\boxtimes}:\Rep_{\epsilon}(\ol{H_1})\times\Rep_{\epsilon}(\ol{H_2})\rightarrow\Rep_{\epsilon}(\ol{H_1\times H_2})\]
	with a twisted induction functor
	\[i=i_{U,\mbm{1}}\circ \varepsilon:\Rep_{\epsilon}(\ol{H_1\times H_2})\rightarrow\Rep_{\epsilon}(\ol{H}),\]
	where $U$ is a unipotent subgroup of $G_r$ being realized as a subgroup of $\ol{H}$, and $\varepsilon$ is a character of $H_1\times H_2$ which acts on $\Rep_{\epsilon}(\ol{H_1\times H_2})$. (Here $\varepsilon$ should not to be confused with $\epsilon$.) Both $U$, $\varepsilon$ and the way of embedding $H_1\times H_2$ into $H$ in the above three cases are defined exactly as in \cite{bernstein1977induced}*{\S 4.12}, and thus we omit the detail.
	
	However, in the above definition  we need to take special care of the following two degenerate cases:
	\begin{enumerate}
		\item[$\bullet$] In case (2), if $r_2=1$, the group $H_1$ is trivial and $\pi_1$ is the character $\epsilon$
		of the group $\ol{H_1}=\mu_n$, then for a genuine representation $\pi_2$ of $\ol{H_2}$, we formally define the EMTP to be
		\[\pi_1\ol{\boxtimes}\pi_2:=[F^{\times}:F^{\times n}]\cdot\Ind_{\ol{\{1\}\times H_2^{(n)}}}^{\ol{\{1\}\times H_2}}(\epsilon\boxtimes \pi_2\rest_{\ol{H_2^{(n)}}})\]
		as a genuine representation of $\ol{H_1\times H_2}=\ol{\{1\}\times H_2}$.
		\item[$\bullet$] In case (3), if $r_2=1$, the group  $H_2$ is trivial, then we formally define the EMTP to be
		\[\pi_1\ol{\boxtimes}\pi_2:=[F^{\times}:F^{\times n}]\cdot\Ind_{\ol{ H_1^{(n)}\times \{1\}}}^{\ol{ H_1\times \{1\}}}( \pi_1\rest_{\ol{H_1^{(n)}}}\boxtimes\epsilon).\]
		as a genuine representation of $\ol{H_1\times H_2}=\ol{H_1\times\{1\}}$. 
	\end{enumerate}
	With these taken into consideration, we compose $\ol{\boxtimes}$ with $i$ to define the product $\ol{\times}$ as in the non-degenerate case.

	\begin{prop}[\cite{bernstein1977induced}*{Propsition 4.13}]  \label{prop:BZ's prop4.13}
		
		Let $\rho\in\Rep_{\epsilon}(\ol{G_{r_1}})$, $\sigma\in\Rep_{\epsilon}(\ol{G_{r_2}})$, $\tau\in\Rep_{\epsilon}(\ol{P_{r_2}})$.
		
		\begin{enumerate}
			\item In $\Rep_{\epsilon}(\ol{P_{r}})$ there exists an exact sequence
			\[
			0 \rightarrow  (\rho\rest_{\ol{P_{r_1}}})\ol{\times}\sigma \xrightarrow{} (\rho\ol{\times}\sigma)\rest_{\ol{P_{r}}}  \xrightarrow{}  \rho\ol{\times}(\sigma\rest_{\ol{P_{r_2}}}) \rightarrow 0.
			\]
			
			\item If $\Omega$ is one of the functors $\Phi^+, \Psi^-,\Phi^-$, then $\Omega(\rho\ol{\times}\tau)\simeq\rho\ol{\times}\Omega(\tau)$ and $\Psi^+(\rho\ol{\times}\sigma)\simeq\rho\ol{\times}\Psi^+(\sigma)$.
			
			\item One has $\Psi^-(\tau\ol{\times}\rho)=\Psi^-(\tau)\ol{\times}\rho$ and there exists an exact sequence 
			\[
			0 \rightarrow  \Phi^-(\tau)\ol{\times}\rho \xrightarrow{} \Phi^-(\tau\ol{\times}\rho)  \xrightarrow{} \Psi^-(\tau)\ol{\times}(\rho\rest_{\ol{P_{r_1}}}) \rightarrow 0.\]
			
			\item Suppose that $r_1>0$. Then for any non-zero $\ol{P_r}$-submodule $\omega\subset\tau\ol{\times}\rho$ we have $\Phi^-(\omega)\neq 0$.
			
		\end{enumerate}
		
	\end{prop}
	
	\begin{proof}

		For (1), we follow the argument in \cite[\S 7.1]{bernstein1977induced} and just replace $\rho\otimes \sigma$ there by $\rho \ol{\boxtimes} \sigma$, then we get an exact sequence \[
		0 \rightarrow i^{\ol{P_r}}_{\ol{P_{r_1}\times G_{r_2}}}((\rho \ol{\boxtimes}\sigma)\rest_{\ol{P_{r_1}\times G_{r_2}}}) \xrightarrow{} (\rho\ol{\times}\sigma)\rest_{\ol{P_{r}}}  \xrightarrow{}  i^{\ol{P_r}}_{\ol{G_{r_1}\times P_{r_2}}}(\rho \ol{\boxtimes}\sigma)\rest_{\ol{G_{r_1}\times P_{r_2}}}) \rightarrow 0.
		\]
		Notice that if $r_1\neq 1$, then $\ol{P_{r_1}} \cap \ol{G^{(n)}_{r_1}}=\ol{P^{(n)}_{r_1}}$ and $\ol{P_{r_1}}/\ol{P^{(n)}_{r_1}} \simeq \ol{G_{r_1}}/\ol{G^{(n)}_{r_1}} \simeq F^{\times}/F^{\times n}$. Thus the group $ \ol{P_{r_1}\times G_{r_2}} $ has only one orbit on $\ol{G^{(n)}_{r_1}\times G^{(n)}_{r_2} } \backslash \ol{G_{r_1}\times G_{r_2}}$. By \cite[Theorem 5.2]{bernstein1977induced} $$(\rho \ol{\boxtimes}\sigma)\rest_{\ol{P_{r_1}\times G_{r_2}}} \simeq (\rho\rest_{\ol{P_{r_1}}})\ol{\boxtimes} \sigma.$$
		If $r_1=1$, then $P_{r_1}$ is trivial and $ \ol{P_{r_1}\times G_{r_2}} $ has $[\ol{G_{r_1}}:\ol{G^{(n)}_{r_1}}]=[F^{\times}:F^{\times n}]$ many orbits on $\ol{G^{(n)}_{r_1}\times G^{(n)}_{r_2} } \backslash \ol{G_{r_1}\times G_{r_2}}$. Each orbit is open and has a representative of the form $(\ol{G^{(n)}_{r_1}\times G^{(n)}_{r_2} })g_1^{-1}$ for $g_1 \in \ol{G_{r_1}\times \{{1\}}}$. Since elements in $\ol{G_{r_1}\times \{1\}}$ commute with elements in $\ol{\{1\}\times G_{r_2}^{(n)}}$, by \cite[Theorem 5.2]{bernstein1977induced} we have 
		$$(\rho \ol{\boxtimes}\sigma)\rest_{\ol{P_{r_1}\times G_{r_2}}}\simeq\bigoplus_{g_1\in\ol{G_{r_1}^{(n)}\times \{1\}}\backslash\ol{G_{r_1}\times \{1\}}}\Ind_{\ol{\{1\}\times G_{r_2}^{(n)}}}^{\ol{\{1\}\times G_{r_2}}}((\rho\rest_{\ol{P_{r_1}}}\boxtimes\sigma\rest_{\ol{G_{r_2}^{(n)}}})^{g_1}) \simeq (\rho\rest_{\ol{P_{r_1}}})\ol{\boxtimes} \sigma.$$
		In parallel, we have  $$(\rho \ol{\boxtimes}\sigma)\rest_{\ol{G_{r_1}\times P_{r_2}}} \simeq \rho\ol{\boxtimes} (\sigma\rest_{\ol{P_{r_2}}}),$$
		regardless of whether $r_2=1$ or not. This finishes the proof of Statement (1). 
		
		For (2), we need to prove the compatibility of the ``partial version'' of the above functors and EMTP, i.e., to show
		\[\,^2\Omega(\rho \ol{\boxtimes} \tau) \simeq \rho \ol\boxtimes \Omega(\tau)\quad\text{and}\quad\,^2\Psi^+(\rho \ol{\boxtimes} \sigma) \simeq \rho \ol\boxtimes \Psi^+(\sigma)\] when $\Omega=\Phi^+, \Psi^-,\Phi^-$.  
		Here $\,^i\Omega$ (resp. $\,^i\Psi^+$) means applying the functor $\Omega$ (resp. $\Psi^+$) to the $i$-th coordinate. 
		%First we consider $\Omega=\Psi^+$ or $\Phi^+$. 
		By \cite[1.9(c)]{bernstein1977induced} we have \[\,^2\Phi^+(\rho \ol{\boxtimes} \tau)\simeq\text{Ind}^{\ol{G_{r_1}\times P_{r_2+1}}}_{\ol{G^{(n)}_{r_1}\times P^{(n)}_{r_2+1} }}\,^2\Phi^+_n(\rho\rest_{\ol{G^{(n)}_{r_1}}}\boxtimes \tau\rest_{\ol{P^{(n)}_{r_2}}})\simeq \text{Ind}^{\ol{G_{r_1}\times P_{r_2+1}}}_{\ol{G^{(n)}_{r_1}\times P^{(n)}_{r_2+1} }}(\rho\rest_{\ol{G^{(n)}_{r_1}}}\boxtimes \Phi^+_n(\tau\rest_{\ol{P^{(n)}_{r_2}}}))=\rho \ol\boxtimes \Phi^+(\tau)\]
		and
		\[\,^2\Psi^+(\rho \ol{\boxtimes} \sigma)\simeq\text{Ind}^{\ol{G_{r_1}\times P_{r_2+1}}}_{\ol{G^{(n)}_{r_1}\times P^{(n)}_{r_2+1} }}\,^2\Psi^+_n(\rho\rest_{\ol{G^{(n)}_{r_1}}}\boxtimes \sigma\rest_{\ol{G^{(n)}_{r_2}}})\simeq \text{Ind}^{\ol{G_{r_1}\times P_{r_2+1}}}_{\ol{G^{(n)}_{r_1}\times P^{(n)}_{r_2+1} }}(\rho\rest_{\ol{G^{(n)}_{r_1}}}\boxtimes \Psi^+_n(\sigma\rest_{\ol{G^{(n)}_{r_2}}}))=\rho \ol\boxtimes \Psi^+(\sigma).\]
		If $\Omega=\Psi^-$, since $\ol{G_{r_1}\times G_{r_2-1}}\cdot V_{r_2}$  has only one orbit on $\ol{G^{(n)}_{r_1}\times P^{(n)}_{r_2} } \backslash \ol{G_{r_1}\times P_{r_2}}$, then using \cite[Theorem 5.2]{bernstein1977induced} we have $\,^2\Psi^-(\rho \ol{\boxtimes} \tau) \simeq \rho \ol\boxtimes \Psi^-(\tau)$. If $\Omega=\Phi^-$, we should consider the orbits of $\ol{G_{r_1}\times P_{r_2-1}}\cdot V_{r_2}$ on $\ol{G^{(n)}_{r_1}\times P^{(n)}_{r_2} } \backslash \ol{G_{r_1}\times P_{r_2}}.$ Depending on $r_2-1\neq 1$ or $r_2-1=1$, we have either one orbit or $[F^{\times}:F^{\times n}]$ many orbits.  
		With a similar discussion to the argument of (1) using \cite[Theorem 5.2]{bernstein1977induced}, we also have $\,^2\Phi^-(\rho \ol{\boxtimes} \tau) \simeq \rho \ol\boxtimes \Phi^-(\tau)$. After this, by \cite[\S 7.2]{bernstein1977induced} we have \[ \Omega(\rho \ol{\times} \tau)=i(^2\Omega(\rho\ol{\boxtimes}\tau))\quad\text{and}\quad \Psi^+(\rho \ol{\times} \sigma)=i(^2\Psi^+(\rho\ol{\boxtimes}\sigma)),\]
		where $i$ is the corresponding twisted induction functor. Thus (2) follows from this formula and the above compatibility result.
		
		For (3), \cite[\S 7.2]{bernstein1977induced} shows that $\Psi^-(\tau \ol{\boxtimes} \rho)=i(^1\Psi^-(\tau \ol{\boxtimes} \rho))$ and there exists an exact sequence \[
		0 \rightarrow i(^1\Phi^-(\tau \ol{\boxtimes} \rho)) \xrightarrow{} \Phi^-(\tau\ol{\times}\rho)  \xrightarrow{} i(^1\Psi^-(\tau \ol{\boxtimes} \rho)|_{\ol{G_{r_2}\times P_{r_1}}}) \rightarrow 0.\] 
		Symmetric to the argument for $\,^2\Omega$ above, we can prove the compatibility for $\,^1\Omega$ with the EMTP, i.e.,
		\[ \,^1\Psi^-(\tau\ol\boxtimes\rho)\simeq\Psi^-(\tau)\ol\boxtimes \rho\quad\text{and}\quad\,^1\Phi^-(\tau\ol\boxtimes\rho)\simeq\Phi^-(\tau).\] 
		Then Statement (3) follows from these facts as well as the isomorphism $(\Psi^-(\tau) \ol{\boxtimes}\rho)\rest_{\ol{G_{r_2}\times P_{r_1}}} \simeq \Psi^-(\tau)\ol{\boxtimes} (\rho\rest_{\ol{P_{r_1}}})$ shown in (1).

		Statement (4) follows from the same argument as in \cite[\S 7.3]{bernstein1977induced}, without essential change.
	\end{proof}
	
	As a direct consequence, we have
	
	\begin{cor}
		Let $\rho\in\Rep(\ol{G_{r_1}})$, $\sigma\in\Rep(\ol{G_{r_2}})$, $\tau\in\Rep(\ol{P_{r_2}})$.
		\begin{enumerate}
			\item If $i\gest1$, then $(\rho\ol{\times}\tau)^{(i)}\simeq\rho\ol{\times}\tau^{(i)}$.
			
			\item If $i\gest 1$, then $(\tau\ol{\times}\rho)^{(i)}$ is glued from $\tau^{(j)}\ol{\times}\rho^{(i-j)}$, $j=1,2,\dots,i$.
			
			\item If $i\gest 0$, then $(\rho\ol{\times}\sigma)^{(i)}$ is glued from $\rho^{(j)}\ol{\times}\sigma^{(i-j)}$, $j=0,1,\dots,i$.
			
		\end{enumerate}
		
	\end{cor}
	
	This ``Lebnitz law'' could be generalized to the product of finitely many representations up to multiplicity, if we confine to finite-length representations.
	
	\begin{cor}\label{cor productderivative}
		
		Let $\rho_i\in\Rep_{\epsilon}^{\rm fl}(\ol{G_{r_i}})$ with $r=r_1+\dots+r_k$. Then for the positive integer $N=[F^{\times}:F^{\times n}]^{k-2}$, the $N$-th multiple 
		of the derivative $(\rho_1\ol{\times}\dots\ol{\times}\rho_k)^{(l)}$ is glued from $N$-th multiples of $\rho_1^{(l_1)}\ol{\times}\dots\ol{\times}\rho_k^{(l_k)}$ with non-negative integers $l_i$ such that $l_1+\dots+l_k=l$.
		
	\end{cor}
	
	\begin{proof}
		
		It follows from the above corollary for $k=2$ and  Proposition \ref{prop BZPasso}. We have to add $(k-2)$'s brackets to reduce to the basic case.
		
	\end{proof}
	
	\begin{rmk}
		
		In the above two corollaries, there exist degenerate cases that should be explained as follows: Consider $\rho_i\in\Rep_{\epsilon}^{\rm fl}(\ol{G_{r_i}})$ with $r_1+\dots+r_k=r$ and possibly $r_i\geq 0$ for some $i$. More precisely, let $\{i_1,\dots,i_s\}$ be the subset of $\{1,\dots,k\}$ such that $r_i\neq 0$, then we formally define
		\[\rho_1\ol{\times}\dots\ol{\times}\rho_k=[F^{\times}:F^{\times n}]^{k-s}\cdot\big(\prod_{i\notin\{i_1,\dots,i_s\}}\dim \rho_i \big)\cdot(\rho_{i_1}\ol{\times}\dots\ol{\times}\rho_{i_s}).\]
		
	\end{rmk}
	
	\subsection{Compatibility of the $n$-cuspidal support with derivatives}
	
	\begin{lm}[\cite{bernstein1977induced}*{Lemma 4.7}] \label{lm:BZ's lemma4.7}
		
		Let $\rho_i\in\Cusp_{\epsilon}(\ol{G_{r_i}})$ and $\pi=\rho_1\ol{\times}\dots\ol{\times}\rho_k$, where $r=r_1+\dots+r_k$.
		
		\begin{enumerate}
			\item If $\sigma$ is an irreducible subquotient of $\pi^{(l)}$ ($l=0,1,\dots,r$), then the $n$-cuspidal support of $\sigma$ is contained in $[\rho_1]+\dots+[\rho_k]$.
			
			\item If $\omega$ is a non-zero subrepresentation of $\pi\rest_{\ol{P_r}}$, and $\omega^{(l)}$ is the highest derivative of $\omega$ and $\sigma$ is an irreducible subrepresentation of $\omega^{(l)}$. Then the $n$-cuspidal support of $\sigma\nu $ is contained in $[\rho_1]+\dots+[\rho_k]$.
			
		\end{enumerate}
		
	\end{lm}
	
	\begin{proof}
		
		For (1), using Corollary \ref{cor productderivative} we know that $\sigma$ is an irreducible subquotient of $\rho_1^{(l_1)}\ol{\times}\dots\ol{\times}\rho_k^{(l_k)}$ for some $l_i$ such that $l_1+\dots+l_k=l$. Using Lemma \ref{L:cuspderivative}, each $l_i$ is either 0 or $r_i$, which implies that the $n$-cuspidal support of $\sigma$ is a multisubset of $[\rho_1]+\dots+[\rho_k]$. For (2), we consider (\emph{cf.} Lemma \ref{lem MTPcontragredient} and \cite{bernstein1977induced}*{2.3.(d), 1.9.(f)})
		\[\pi=\rho_1\ol{\times}\dots\ol{\times}\rho_k\quad\text{and}\quad \tilde{\pi}:=\pi^{\vee}\nu=(\rho_1\ol{\times}\dots\ol{\times}\rho_k)^{\vee}\nu\simeq\rho_1^{\vee}\nu\ol{\times}\dots\ol{\times}\rho_k^{\vee}\nu.\]
		Restricting to $\ol{P_{r}}$, there is a  $\delta_{\ol{P_r}}^{-1}$-pairing of $\pi\rest_{\ol{P_{r}}}$ and $\tilde{\pi}\rest_{\ol{P_{r}}}$ which is non-degenerate with respect to $\pi\rest_{\ol{P_{r}}}$. Restricting to $\omega$, there is a $\delta_{\ol{P_r}}^{-1}$-pairing of $\omega$ and $\tilde{\pi}\rest_{\ol{P_{r}}}$ which is non-degenerate with respect to $\omega$. Using Proposition \ref{prop derivativepairing}, there is a  $\mbm{1}$-pairing of $\omega^{(l)}$ and $\tilde{\pi}^{(l)}$ which is non-degenerate with respect to $\omega^{(l)}$. This means that $\sigma^{\vee}$ is an irreducible subquotient of $\tilde{\pi}^{(l)}$. Since each $\rho_i^{\vee}\nu$ is cuspidal, using (1) the $n$-cuspidal support of $\sigma^{\vee}$ is in $[\rho_1^{\vee}\nu]+\dots+[\rho_k^{\vee}\nu]$, thus the $n$-cuspidal support of $\sigma\nu $ is  in $[\rho_1]+\dots+[\rho_k]$.
		
	\end{proof}
	
	\section{Bernstein--Zelevinsky derivatives for Kazhdan--Patterson and Savin covers}\label{section resultKPS}
	
	In this section, we take $\ol{G_r}$ to be either a KP-cover or an S-cover of $G_r$, $r\gest 1$. 
	
	\subsection{Zelevinsky classification for Kazhdan--Patterson and Savin covers}

	First, for $\rho \in {\rm Cusp}_\epsilon(\ol{G_{r_0}})$, there exists a unique positive real number $l(\rho) \in \R_{>0}$, such that the related parabolic induction $(\rho\times\rho\nu^{s})_{\omega}$ is reducible if and only if $s=\pm l(\rho)/n$. In \cite{zou2025simple}*{\S 6.3, \S 9.5} and under the tame condition, the number $l(\rho)$ is explicitly calculated, which in particular is a positive integer dividing $n$. Thus, we write 
	$$n(\rho):=\frac{n}{l(\rho)} \in \N_{\gest 1},$$
	which also gives an unramified character
	$\nu_{\rho}=\nu^{1/n(\rho)}$  of $G_{r_0}$.
	
	\begin{rmk}
		
		For a KP-cover, an alternative way of characterizing $l(\rho)$ is that it is the cardinality of the cuspidal support of the metaplectic lift of $\rho$ as a discrete series representations of $G_{r_0}$ (\emph{cf.} \cite{zou2022metaplectic}*{Proposition 3.14}). In particular, $l(\rho)$ divides $r_0$ as well.
		
	\end{rmk}
	
	We consider segments and multisegments. Recall that a \emph{segment} $\Delta=[a,b]_{\rho}$ is of the form
	\[\begin{cases}
		\{\rho\nu_{\rho}^a,\rho\nu_{\rho}^{a+1}\dots,\rho\nu_{\rho}^b\} &  \text{ if $\ol{G_r}$ is an S-cover}, \\
		\{[\rho\nu_{\rho}^a],[\rho\nu_{\rho}^{a+1}]\dots,[\rho\nu_{\rho}^b]\} &  \text{ if $\ol{G_r}$ is a KP-cover},
	\end{cases}  \]
	where $\rho\in\Cusp_{\epsilon}(\ol{G_{r_0}})$ for some $r_0\gest 1$, and $a\lest b$ are integers. In our convention, $[a,a-1]_{\rho}$ is the empty set. A \emph{multisegment}  is a finite multiset of segments. For two segments, we may discuss whether they are linked or whether one precedes the other as in \cite{kaplan2022classification}*{\S 7.1}. More precisely, two segments $\Delta_1,\Delta_2$ are called \emph{linked} if their union $\Delta_1\cup\Delta_2$ is a strictly larger segment. In this case, we may write $\Delta_1=[a_1,b_1]_{\rho}$ and $\Delta_2=[a_2,b_2]_{\rho}$ for some common $\rho\in\Cusp_{\epsilon}(\ol{G_{r_0}})$. Then, either $\Delta_1$ \emph{precedes} $\Delta_2$ meaning that $a_1<a_2$, $b_1<b_2$ and $b_1\gest a_2-1$; or $\Delta_2$ \emph{precedes} $\Delta_1$ meaning that $a_2<a_1$, $b_2<b_1$ and $b_2\gest a_1-1$.
	
	Given a segment $\Delta=[a,b]_{\rho}$, consider the Bernstein--Zelevinsky product
	\[(\rho\nu_{\rho}^a\times\rho\nu_{\rho}^{a+1}\times\dots\times\rho\nu_{\rho}^b)_{\omega}\]
	which has a unique irreducible subrepresentation denoted by $Z([a,b]_{\rho})_{\omega}$ and a unique irreducible quotient denoted by $L([a,b]_{\rho})_{\omega}$. For KP-covers, the definition here depends only on the choice of a compatible central character $\omega$ but not the representative $\rho\nu_{\rho}^i$ in the $n$-equivalence classes $[\rho\nu_{\rho}^i]$. We write $Z(\Delta)$ (resp. $L(\Delta)$) for the $n$-equivalence class of $Z([a,b]_{\rho})_{\omega}$ (resp. $L([a,b]_{\rho})_{\omega}$), or by abuse of notation, for any randomly chosen representative in that class. For S-cover, since this $\omega$ is only symbolic, we just write $Z(\Delta)=Z([a,b]_{\rho})_{\omega}$ and $L(\Delta)=L([a,b]_{\rho})_{\omega}$ as actual irreducible representations. 
	
	Now given a multisegment $\mfr m=\{\Delta_1,\dots,\Delta_k\}$, we may rearrange these segments such that $\Delta_i$ does not precede $\Delta_j$ for $1\lest i<j\lest k$. Then the Bernstein--Zelevinsky product
	\[(Z(\Delta_1)\times\dots\times Z(\Delta_k))_{\omega}\quad(\text{resp.}\ (L(\Delta_1)\times\dots\times L(\Delta_k))_{\omega})\]
	has a unique irreducible subrepresentation (resp. quotient) denoted by
	$Z(\mfr m)_{\omega}$ (resp. $L(\mfr m)_{\omega}$). When ranging over all such $\mfr m$ and compatible central characters $\omega$ in the KP-cover case, these $Z(\mfr m)_{\omega}$ (resp. $L(\mfr m)_{\omega}$) uniquely represent all possible genuine irreducible representations of $\ol{G_r}, r\gest 1$. In the KP-covers case, we write $Z(\mfr m)$ (resp. $L(\mfr m)$) for the $n$-equivalence class of $Z(\mfr m)_{\omega}$ (resp. $L(\mfr m)_{\omega}$), or by abuse of notation, any randomly picked represenative in that class. For S-cover, since $\omega$ is only symbolic, we simply write $Z(\mfr m)=Z(\mfr m)_{\omega}$ (resp. $L(\mfr m)=L(\mfr m)_{\omega}$) as irreducible representations.
	
	For later use, we define 
	\[\Pi(\mfr m)_{\omega}:=(Z(\Delta_1)\times\dots\times Z(\Delta_k))_{\omega}.\]
	In the KP-cover case, if the choice of central character $\omega$ is immaterial, then we simply fix a random  $\omega$ and write $\Pi(\mfr m)$ instead by omitting the subscript $\omega$. We also define the extended  Bernstein--Zelevinsky product 
	\[\ol \Pi(\mfr m)=Z(\Delta_1)\ol\times\dots\ol\times Z(\Delta_k).\] 
	Using Proposition \ref{prop irredBZPandEBZP}, $\Pi(\mfr m)$ is a direct summand of $\ol \Pi(\mfr m)$.

	In the case of KP-covers, given multisegments $\mfr m_1,\dots,\mfr m_k$ of the same size, we write 
	\[Z(\mfr m_1)+\dots+Z(\mfr m_k)\quad (\text{resp.}\ k\cdot Z(\mfr m_1))\]
	for the direct sum of the $k$-many $n$-equivalence classes (resp. $k$-isotypic sum), and also for the direct sum of $k$-many randomly chosen irreducible representations in the corresponding classes, whenever the choice of representations in each $n$-equivalence class is immaterial for us. Moreover, equality like
	\[Z(\mfr m_1)+\dots+Z(\mfr m_k)=Z(\mfr m_1')+\dots+Z(\mfr m_k')\]
	means that there are compatible central characters $\omega_i$, $\omega_i'$, $i=1,\dots,k$ such that
	\[Z(\mfr m_1)_{\omega_1}\oplus\dots\oplus Z(\mfr m_k)_{\omega_k}\cong Z(\mfr m_1')_{\omega_1'}\oplus\dots\oplus Z(\mfr m_k')_{\omega_k'}.\]
	Such compressed notation will help us to simplify the statement and argument of results below.

	We also recall the following proposition for later purpose.
	
	\begin{prop}[\cite{kaplan2022classification}*{Proposition 7.2}]\label{prop JacquetZL}
		
		Given $\rho\in\Cusp_{\epsilon}(\ol{G_{r_0}})$ and $a\lest b$ such that $r=r_0(b-a+1)$. Write $N_{r-k,k}$ for the standard unipotent radical associated with the composition $(r-k,k)$, then
		\[J_{N_{r-k,k}}(Z([a,b]_{\rho})_{\omega})=\begin{cases} (Z([a,b-s]_{\rho})\boxtimes Z([b-s+1,b]_{\rho}))_{\omega}\quad &\text{if }k=r_0s,\\
			0 \quad &\text{if }r_0\ \text{does not divide }k;
		\end{cases}\]
		and
		\[J_{N_{r-k,k}}(L([a,b]_{\rho})_{\omega})=\begin{cases} (L([a+s,b]_{\rho})\boxtimes L([a,a+s-1]_{\rho}))_{\omega}\quad &\text{if }k=r_0s,\\
			0 \quad &\text{if }r_0\ \text{does not divide }k.
		\end{cases}\]

	\end{prop}
	
	\subsection{Calculating derivatives for standard modules $Z(\Delta)$ and $L(\Delta)$}
	
	In this subsection, we calculate the derivatives of $Z(\Delta)$ and $L(\Delta)$, where 
	$\Delta=[a,b]_{\rho}$ is a segment. To that end, we need to assume the tame condition $\gcd(n,p)=1$ as well.
	
	Before confining to the more specific case, we first show the following general result:
	
	\begin{prop}\label{prop WhdimBZprod}
		
		Let $r=r_1+\dots+r_k$ and $\rho_i\in\Irr_{\epsilon}(\ol{G_{r_i}})$. Then
		\[
		\dim \Wh ((\rho_1\times\dots\times\rho_k)_{\omega}) =
		\begin{cases}
			\prod_{i=1}^{k}\dim(\Wh(\rho_i))	\quad& \text{for an S-cover},\\
			\big(\prod_{i=1}^{k}\dim(\Wh(\rho_i))d_{r_i}\big)/d_r\quad& \text{for a KP-cover},
		\end{cases}\]where $d_r=\gcd(n,2r\mbf{a}-r+1)$ is defined in \eqref{eq drgcd}.
	\end{prop}
	\begin{proof}
		
		First, we have that
		\[\dim\Wh((\rho_1\times\dots\times\rho_k)_{\omega})=\dim\Wh((\rho_1\boxtimes\dots\boxtimes\rho_k)_{\omega}),\]
		which follows from the heredity of Whittaker models in \cite{banks1998heredity}. If we consider the S-cover, then from the construction of tensor product we have
		\[\dim\Wh((\rho_1\boxtimes\dots\boxtimes\rho_k)_{\omega})=\prod_{i=1}^{k}\dim(\Wh(\rho_i)).\]
		If we consider the KP-cover, then by taking $\rho=(\rho_1\boxtimes\dots\boxtimes\rho_k)_{\omega}$ in \eqref{eq WhitdimBD} and using the equality  
		\[(\prod_{i=1}^km_i)/m=(\prod_{i=1}^kd_{r_i})/d_r\] deduced from \cite{zou2022metaplectic}*{Proposition 2.3}, we finish the proof as well.
		
	\end{proof}
	
	We recall the following result on Whittaker dimensions of $Z(\Delta)$ and $L(\Delta)$ under the assumption $\gcd(p,n)=1$. In particular, Whittaker dimension is independent of the representative in an $n$-equivalence class.
	
	\begin{prop}[\cite{zou2025gelfand}*{Section 10}]\label{prop WhittakerZL}
		Assume $\gcd(n,p)=1$. Then we have
		\[
		\dim \Wh (Z(\Delta))=
		\begin{cases} 
			\binom{n(\rho)}{b-a+1} &\quad\text{for an S-cover},\\
			\binom{n(\rho)}{b-a+1}/d_r &\quad\text{for a KP-cover};
		\end{cases}\]
		and
		\[
		\dim \Wh (L(\Delta))=
		\begin{cases} 
			\binom{n(\rho)+b-a}{b-a+1} &\quad\text{for an S-cover},\\
			\binom{n(\rho)+b-a}{b-a+1}/d_r &\quad\text{for a KP-cover}.
		\end{cases}\]
		
	\end{prop}
	
	Now we discuss derivatives.
	
	\begin{prop} \label{prop: der_of_ZLDelta}
		Assume $\gcd(n,p)=1$. Regarding the $k$-th derivatives, we have
		\[
		Z(\Delta)^{(k)} \simeq 
		\begin{cases} 
			\binom{n(\rho)}{s} \cdot Z([a,b-s]_{\rho})\quad &\text{if }k=r_0s\ \text{for an S-cover},\\
			\frac{d_{r-k}}{d_r}\binom{n(\rho)}{s} \cdot Z([a,b-s]_{\rho})\quad &\text{if }k=r_0s\ \text{for a KP-cover},\\
			0 \quad &\text{if }r_0\ \text{does not divide }k;
		\end{cases}\]
		and 
		\[
		L(\Delta)^{(k)} \simeq 
		\begin{cases} 
			\binom{n(\rho)+s-1}{s} \cdot L([a+s,b]_{\rho})\quad &\text{if }k=r_0s\ \text{for an S-cover},\\
			\frac{d_{r-k}}{d_r}\binom{n(\rho)+s-1}{s} \cdot L([a+s,b]_{\rho})\quad &\text{if }k=r_0s\ \text{for a KP-cover},\\
			0 \quad &\text{if }r_0\ \text{does not divide }k.
		\end{cases}\]
		Here  in the KP-cover case, the isomorphism above is only up to $n$-equivalence class.
	\end{prop}
	
	\begin{proof}
		
		Recall that we have $(\cdot)^{(k)}=\partial^k\circ J_{N_{r-k,k}}$. First, in the  S-cover case, for $\pi_1\in\Rep_{\epsilon}(\ol{G_{r-k}})$ and $\pi_2\in\Rep_{\epsilon}(\ol{G_{r-k}})$, we have $\partial^k(\pi_1\boxtimes\pi_2)=\dim(\Wh(\pi_2))\cdot\pi_1$. The result follows from Propositions \ref{prop JacquetZL} and \ref{prop WhittakerZL}. For KP-covers, by using the calculation in \S \ref{subsubsection derivativeBD}, formula \eqref{eq tensormult}, Propostions \ref{prop JacquetZL} and \ref{prop WhittakerZL}, we also get the desired result.
		
	\end{proof}
	
	\begin{cor}\label{cor derivativeZ}
		The highest derivative of $Z(\Delta)$ is of degree $\min(n(\rho)r_0,(b-a+1)r_0)$, and is isomorphic to 
		\[\begin{cases}
			Z([a,b-n(\rho)])_{\rho}\quad&\text{if }b-a+1\gest n(\rho),\\
			\binom{n(\rho)}{b-a+1}\cdot\epsilon \quad&\text{if }b-a+1\lest n(\rho)\ \text{for an S-cover},\\
			\binom{n(\rho)}{b-a+1}/d_r\cdot\epsilon \quad&\text{if }b-a+1\lest n(\rho)\ \text{for an KP-cover}.
		\end{cases}\]
		Again, in the KP-cover case, the isomorphism above is only up to $n$-equivalence class.
	\end{cor}
	
	\begin{proof}
		In view of Proposition \ref{prop: der_of_ZLDelta}, it suffices to check the following: for a KP-cover, if $b-a+1>n(\rho)$, then $d_r=d_{r-n(\rho)r_0}$. However, this follows easily from the fact that $n$ divides $n(\rho)r_0=nr_0/l(\rho)$ (note that, for a KP-cover, we have $l(\rho)$ divides $r_0$) and the definition of $d_r$.
	\end{proof}
	
	\begin{cor}\label{cor generalizedWhittseg}
		Let $\lambda=(r_1,\dots,r_k)$ be a composition of $r$. Then the $\lambda$-semi-Whittaker model of $Z(\Delta)\in \Irr_{\epsilon}(\ol{G_r})$ is non-zero if and only if each $r_i$ is divisible by $r_0$ and smaller than and equal to $n(\rho)r_0$.
		
	\end{cor}
	
	\begin{proof}
		
		It follows from Propositions \ref{prop JacquetZL}, \ref{prop WhittakerZL} and also the equalities in \eqref{eq generalizedWhit} and \eqref{eq WhitdimBD}.
		
	\end{proof}

	\section{Highest Bernstein--Zelevinsky derivatives}\label{section HighDegWhit}
	In this section, we continue to take $\ol{G_r}$ to be a KP-cover or S-cover, and will define for each $\pi \in \Irr_\epsilon(\ol{G_r})$ a partition $\lambda_{\pi}$ of $r$, such that $\pi$ has non-zero $\lambda_\pi$-semi-Whittaker model. This will give a lower bound of the wavefront set of $\pi$, to be discussed in \S  \ref{section WFandLLC}.
	
	\subsection{Main result on highest derivatives}
	
	Given a smooth representation $\pi$ of $\ol{G_r}$, let $\pi_1=\pi^{(k_1)}, k_1 \gest 1$ be the highest derivative of $\pi$, as a  representation of $\ol{G_{r-k_1}}$. Assume that $\pi_{s}$ and $k_s$ are defined for some positive integer $s$. Let $\pi_{s+1}=\pi_s^{(k_{s+1})}$ be the highest derivative of $\pi_{s}$ as a smooth representation of $\ol{G_{r-\sum_{i=1}^{s+1}k_i}}$, where $k_{s+1}$ is the related degree. Such a procedure stops at some integer $s_{\pi}$. Let 
	$$\lambda_{\pi}:=(k_1,k_2,\dots,k_{s_{\pi}})$$
	be the composition of $r$ obtained in this way. If $\pi$ is of finite length, then so is each $\pi_i$.
	
	Let $\mfr m=\Delta_1+\dots+\Delta_k$ be a multisegment with $\Delta_{i}=[0,l_i-1]_{\rho_i}$, $\rho_i\in\Cusp_{\epsilon}(\ol{G_{r_i^0}})$, $l_i\gest 1$ for each $1\lest i\lest k$.  Write $r_i=l_ir_i^0$ for each $i=1,\dots,k$ and $r=\sum_{i=1}^k r_i=\sum_{i=1}^kl_ir_i^0$. Write 
	\[k_{\mfr m}:=\sum_{i=1}^kr_i^0 \cdot\min(l_i,n(\rho_i)).\]
	
	The following result parallels that for linear $G_r$, as given by Zelevinsky \cite[\S 8]{zelevinsky1980induced}.
	
	\begin{thm}\label{thm highestderivative}
		
		Let $\ol{G_r}$ be an S-cover or a KP-cover.
		\begin{enumerate}
			\item The highest derivative of $Z(\mfr m)$ is of degree
			$k_1=k_{\mfr m}$.
			
			\item If $\Delta_i$ does not precede $\Delta_j$ for $1\lest i<j\lest k$, then the socle of $\Pi(\mfr m)^{(k_1)}$ is a finite multiple of $Z(\mfr m^-)$, where
			\[\mfr m^-=[0,l_1-1-\min(l_1,n(\rho_1))]_{\rho_1}+\dots+[0,l_k-1-\min(l_k,n(\rho_k))]_{\rho_k}.\]

			\item If $l_i\gest n(\rho_i)$ for each  $i=1,\dots,k$, then $Z(\mfr m)^{(k_1)}$ is isomorphic to $Z(\mfr m^-)$.

		\end{enumerate}
		
	\end{thm}
	
	For $Z(\mfr{m})$, we write
	\begin{equation} \label{D:lambda}
		\lambda_{\mfr m}:=\lambda_{Z(\mfr{m})}=(k_1, k_2,\dots, k_{s_{\mfr m}})
	\end{equation}
	for the composition associated with $Z(\mfr m)$. Using the above theorem and \eqref{E:r-BZ}, we have the following result.
	
	\begin{cor}\label{cor generalizedWhitt}
		
		\begin{enumerate}
			
			\item $\lambda_{\mfr m}$ is a partition of $r$ with $k_1\gest k_2\gest \dots \gest k_{s_{\mfr m}}$.
			
			\item The representation $Z(\mfr m)$ has a non-zero $\lambda_\mfr{m}$-semi-Whittaker model.
			
			\item If moreover each $l_i$ is divisible by $n(\rho_i)$, then the space of $\lambda_\mfr{m}$-semi-Whittaker models of $Z(\mfr m)$ is of dimension one.
			
		\end{enumerate}
		
	\end{cor}
	
	\subsection{Homogeneity and socle}
	
	Given a representation $\tau$ of $\ol{P_r}$, we let $k_{\tau}$ be the largest integer $k$ such that $\tau^{(k)}\neq 0$. Similarly, given a representation $\pi$ of $\ol{G_r}$, we set $k_{\pi}:=k_{\pi\rest_{\ol{P_r}}}$. 
	
	A representation $\tau\in\Rep(\ol{P_r})$ is called \emph{homogeneous} if for any non-zero subrepresentation $\sigma$ of $\tau$, we have $k_{\sigma}=k_{\tau}$. Similarly, a representation $\pi\in\Rep_\epsilon(\ol{G_r})$ is called \emph{homogeneous} if it is so for $\pi\rest_{\ol{P_r}}$ (\cite{zelevinsky1980induced}*{\S 5.1}). We have the following proposition whose proof can be applied directly here.
	
	\begin{prop}\cite{zelevinsky1980induced}*{Proposition 5.3}\label{prop prodhomogeneous} 
		Let $\pi\in\Rep_{\epsilon}(\ol{G_{r_1}})$ be homogeneous and $\tau\in\Rep_{\epsilon}(\ol{G_{r_2}})$. Then $\pi\ol{\times}\tau\in\Rep_{\epsilon}(\ol{G_{r_1+r_2}})$ is homogeneous.
	\end{prop}
	
	\begin{prop}
		
		Let $\mfr m=\Delta_1+\dots+\Delta_k$ be a multisegment. Then we have
		\[k_{\Pi(\mfr m)}=k_{\ol \Pi(\mfr m)}=k_{Z(\Delta_1)}+\dots+k_{Z(\Delta_k)}=\sum_{i=1}^kr_i^0\min(l_i,n(\rho_i)).\]
		
	\end{prop}
	
	\begin{proof}
		
		The first two equalities follow from the Lebnitz law (\emph{cf.} Corollary \ref{cor productderivative}) and the relation between the extended and non-extended Bernstein--Zelevinsky product (\emph{cf.} Proposition \ref{prop irredBZPandEBZP}). The third equality follows from Corollary \ref{cor derivativeZ}.
		
	\end{proof}
	
	For a segment $\Delta=[0,l-1]_{\rho}$ with $\rho\in\Cusp_{\epsilon}(\ol{G_{r_{0}}})$, using Corollary \ref{cor derivativeZ} we have $k_{\Delta}=\min(lr_0,n(\rho)r_0)$. Then we define $\Delta^-=[0,l-1-\min(l,n(\rho))]_{\rho}$. For a multisegment $\mfr m=\Delta_1+\dots+\Delta_k$, we define $\mfr m^-:=\Delta_1^-+\dots+\Delta_k^-$.
	
	\begin{prop}\label{prop homogeneous}
		
		Let $\mfr m=\Delta_1+\dots+\Delta_k$ be such that $\Delta_i$ does not precede $\Delta_j$ for $1\lest i<j\lest k$ and $[1-\min(l_i,n(\rho_i))+n(\rho_i),n(\rho_i)]_{\rho_i'}\cap \bigcup_{j=i+1}^{k}\Delta_{j}=\emptyset$ for each $i=1,\dots,k-1$, where $\rho_i':=\rho_i\nu^{l_i-1}_{\rho_i}$ denotes the representation at the right endpoint of $\Delta_i$. Then the following holds:
		
		\begin{enumerate}
			
			\item $\ol \Pi(\mfr m)$ is homogeneous, and thus $k_{\mfr m}=k_{\ol \Pi(\mfr m)}$.
			
			\item The socle of the highest derivative $\Pi(\mfr m)^{(k_{\mfr m})}$ of $\Pi(\mfr m)$ is $c_{\mfr m}\cdot Z(\mfr m^-)$ with
			\[c_{\mfr m}:=\begin{cases} \prod_{i=1}^k\binom{n(\rho_i)}{\min(l_i,n(\rho_i))}&\quad \text{for S-covers;}\\
				(d_{r-k_{\mfr m}}/d_r)\cdot\prod_{i=1}^k\binom{n(\rho_i)}{\min(l_i,n(\rho_i))}&\quad \text{for KP-covers}.
			\end{cases}\]
			Moreover, the number $c_{\mfr m}$ is exactly the total multiplicity of $Z(\mfr m^-)$ in $\Pi(\mfr m)^{(k_{\mfr m})}$ as subquotients.
			
			\item In particular, the socle of the highest derivative $Z(\mfr m)^{(k_{\mfr m})}$ of $Z(\mfr m)$ is a finite isotypic sum of $Z(\mfr m^-)$.
			
		\end{enumerate}
		
	\end{prop}

	\begin{proof}
		
		The proof of Statement (1) is an adaptation of \cite{bernstein1977induced}*{Theorem 4.11} as well as \cite{emerton2014local}*{Proposition 4.3.2}. We give some details.

		Recall $\ol{\Pi}(\mfr{m})=Z(\Delta_1)\ol{\times}\dots\ol{\times} Z(\Delta_k)$. It suffices to show that for any non-zero irreducible subrepresentation $\sigma$ in $\ol{\Pi}(\mfr{m})|_{\ol{P_r}}$, we have $k_{\sigma}=k_{\ol{\Pi}(\mfr{m})}=k_{\mfr m}$. The strategy is to apply induction on $k$. When $k=1$, this is from Proposition \ref{prop: der_of_ZLDelta}.
		
		Since the operator $\ol{\times}$ is associative up to scalar, we know $\sigma$ is an irreducible subrepresentation of $(Z(\Delta_1)\ol{\times}\ol{\Pi}(\mfr{m}^{\prime}))|_{\ol{P_r}}$, where $\mfr{m}^{\prime}=\Delta_2+\dots+\Delta_k$. By Proposition \ref{prop:BZ's prop4.13}, we have the exact sequence of $\ol{P_r}$-representations:
		\[
		0 \rightarrow (Z(\Delta_1)\rest_{\ol{P_{r_1}}})\ol{\times}\ol{\Pi}(\mfr{m}^\prime) \xrightarrow{} (Z(\Delta_1)\ol{\times}\ol{\Pi}(\mfr{m}^\prime))\rest_{\ol{P_{r}}}  \xrightarrow{} Z(\Delta_1)\ol{\times}(\ol{\Pi}(\mfr{m}^\prime)\rest_{\ol{P_{r-r_1}}}) \rightarrow 0.
		\]
		Then $\sigma$ is a subrepresentation of $(Z(\Delta_1)\rest_{\ol{P_{r_1}}})\ol{\times}\ol{\Pi}(\mfr{m}^\prime)$ or $Z(\Delta_1)\ol{\times}(\ol{\Pi}(\mfr{m}^\prime)\rest_{\ol{P_{r-r_1}}})$.
		
		First, we claim that $\sigma$ can not be in $Z(\Delta_1)\ol{\times}(\ol{\Pi}(\mfr{m}^\prime)\rest_{\ol{P_{r-r_1}}})$. If it were, we will have $\sigma^{(k)}$ embeds in $(Z(\Delta_1)\ol{\times}(\ol{\Pi}(\mfr{m}^\prime)\rest_{\ol{P_{r-r_1}}}))^{(k)}=Z(\Delta_1)\ol{\times}(\ol{\Pi}(\mfr{m}^\prime)\rest_{\ol{P_{r-r_1}}})^{(k)}$. In particular, if $\sigma^{(k)}$ is the highest derivative of $\sigma$, then its $n$-cuspidal support contains $\rho^{\prime}_1$ and by Lemma \ref{lm:BZ's lemma4.7} $\rho^{\prime}_1\nu$ is contained in $\Delta_1+...+\Delta_k$. But this is impossible due to our choice of ordering.
		
		Thus, $\sigma$ must be in $(Z(\Delta_1)\rest_{\ol{P_{r_1}}})\ol{\times}\ol{\Pi}(\mfr{m}^\prime)$. It follows from Proposition \ref{prop:BZ's prop4.13} and Proposition \ref{prop: der_of_ZLDelta} that for $i<r^0_1$ we have $$ (\Phi^{-})^{i}((Z(\Delta_1)\rest_{\ol{P_{r_1}}})\ol{\times}\ol{\Pi}(\mfr{m}^\prime))=(\Phi^{-})^{i}((Z(\Delta_1)\rest_{\ol{P_{r_1}}}))\ol{\times}\ol{\Pi}(\mfr{m}^\prime)$$ and \small \[
		(\Phi^{-})^{r^0_1}((Z(\Delta_1)\rest_{\ol{P_{r_1}}}))\ol{\times}\ol{\Pi}(\mfr{m}^\prime) \hookrightarrow{} (\Phi^{-})^{r^0_1}((Z(\Delta_1)\rest_{\ol{P_{r_1}}})\ol{\times}\ol{\Pi}(\mfr{m}^\prime)) \twoheadrightarrow{}  Z([0,l_1-2]_{\rho_1})\ol{\times}(\ol{\Pi}(\mfr{m}^\prime)\rest_{\ol{P_{r-r_1}}}).
		\]\normalsize Since $\sigma$ has at least one nonzero derivative, we have $(\Phi^{-})^{r^0_1-1}(\sigma) \neq 0$ and by Proposition \ref{prop:BZ's prop4.13}.(4), $(\Phi^{-})^{r^0_1}(\sigma) \neq 0$. 
		
		Let $\sigma_1$ be a non-zero irreducible subrepresentation of $(\Phi^{-})^{r^0_1}(\sigma)$. If $\sigma_1$ is contained in $Z([0,l_1-2]_{\rho_1})\ol{\times}(\ol{\Pi}(\mfr{m}^\prime)\rest_{\ol{P_{r-r_1}}})$, then by a similar argument $\rho^{\prime}_1\nu^{(n(\rho_1)-1)/n(\rho_1)}$ is contained in $\Delta_1+...+\Delta_k$. This is possible only when $n(\rho_1)=1$. We assume $n(\rho_1)=1$. In this case, if $\sigma_1$ is contained in $(\Phi^{-})^{r^0_1}((Z(\Delta_1)\rest_{\ol{P_{r_1}}}))\ol{\times}\ol{\Pi}(\mfr{m}^\prime)$, we will have $(\Phi^{-})^{r^0_1}(\sigma_1) \neq 0$ following a similar process. But since $(Z(\Delta_1))^{(2r^0_1)}=0$, we see $(\Phi^{-})^{r^0_1}(\sigma_1)$ is contained in $(\Phi^{-})^{2r^0_1}((Z(\Delta_1)\rest_{\ol{P_{r_1}}})\ol{\times}\ol{\Pi}(\mfr{m}^\prime))=(\Phi^{-})^{2r^0_1}((Z(\Delta_1)\rest_{\ol{P_{r_1}}}))\ol{\times}\ol{\Pi}(\mfr{m}^\prime)$. Proceeding in this fashion we find that $(\Phi^{-})^{jr^0_1}(\sigma) \neq 0$ for any $j$, which is impossible, so in this case $\sigma_1$ is contained in $Z([0,l_1-2]_{\rho_1})\ol{\times}(\ol{\Pi}(\mfr{m}^\prime)\rest_{\ol{P_{r-r_1}}})$. This latter representation is homogeneous by Proposition \ref{prop prodhomogeneous} and the fact that  $\ol{\Pi}(\mfr{m}^\prime)$ is homogeneous from the inductive
		hypothesis. Then $\sigma_1^{(k_{\mfr{m}^{\prime}})} \neq 0$ and so is $ \sigma^{(k_{\mfr{m}})}$, since $\sigma^{(k_{\mfr{m}^\prime})}_1 \subset\sigma^{(k_{\mfr{m}})}= ((\Phi^{-})^{r^0_1}(\sigma))^{(k_{\mfr{m}^{\prime}})}$.

		If $n(\rho_1)\neq 1$, as mentioned before we have $\sigma_1 \subset (\Phi^{-})^{r^0_1}((Z(\Delta_1)\rest_{\ol{P_{r_1}}}))\ol{\times}\ol{\Pi}(\mfr{m}^\prime)$. We take $\sigma_i$ to be a non-zero irreducible subrepresentation of $(\Phi^-)^{r^0_1}(\sigma_{i-1})$ if possible and apply the argument in the last paragraphs to $\sigma_{i-1}$ for $i=2,\dots,\text{min}(l_1,n(\rho_1))$. More precisely, write 
		$$\Gamma_i:=(\Phi^{-})^{ir^0_1}((Z(\Delta_1)\rest_{\ol{P_{r_1}}}))\ol{\times}\ol{\Pi}(\mfr{m}^\prime)$$
		and  assume that $\sigma_{i-1} \subset \Gamma_{i-1}$. Using Proposition \ref{prop:BZ's prop4.13} and Proposition \ref{prop: der_of_ZLDelta} that for $j<r^0_1$ we have 
		\[(\Phi^{-})^{j}(\Gamma_{i-1})=(\Phi^{-})^{r_0^1(i-1)+j}((Z(\Delta_1)\rest_{\ol{P_{r_1}}}))\ol{\times}\ol{\Pi}(\mfr{m}^\prime)\]
		and thus
		\[
		\Gamma_i=(\Phi^{-})^{ir^0_1}((Z(\Delta_1)\rest_{\ol{P_{r_1}}}))\ol{\times}\ol{\Pi}(\mfr{m}^\prime) \hookrightarrow{} (\Phi^{-})^{r^0_1}(\Gamma_{i-1}) \twoheadrightarrow{}  Z([0,l_1-i-1]_{\rho_1})\ol{\times}(\ol{\Pi}(\mfr{m}^\prime)\rest_{\ol{P_{r-r_1}}}).
		\]
		Then $0\neq (\Phi^{-})^{r^0_1-1}(\sigma_{i-1})\subset (\Phi^{-})^{r_1^0-1}(\Gamma_{i-1})$ and by  Proposition \ref{prop:BZ's prop4.13}.(4), we get $(\Phi^{-})^{r^0_1}(\sigma_{i-1}) \neq 0$. 
		Then $\sigma_i$ can indeed be chosen to be a subrepresentation of either $\Gamma_i$ or $Z([0,l_1-i-1]_{\rho_1})\ol{\times}(\ol{\Pi}(\mfr{m}^\prime)\rest_{\ol{P_{r-r_1}}})$. If furthermore $i< \text{min}(l_1,n(\rho_1))$, then
		by considering the $n$-cuspidal support and Lemma \ref{lm:BZ's lemma4.7}, the latter case implies that $\rho_1'\nu^{(n(\rho_1)-i)/n(\rho_1)}$ is contained in $\Delta_1+\dots+\Delta_k$, which is impossible from our assumption. Thus we must have that $\sigma_i\subset \Gamma_i$. Now consider the case where $i=\text{min}(l_1,n(\rho_1))$. In this case we cannot have $\sigma_{\text{min}(l_1,n(\rho_1))r^0_1}\subset\Gamma_{\text{min}(l_1,n(\rho_1))r^0_1}$, otherwise as in the $n(\rho_1)=1$ case we may deduce that $(\Phi^{-})^{jr_1^0}(\sigma_{\text{min}(l_1,n(\rho_1))r^0_1})\neq 0$ for any $j$, which is absurd. Thus we have $\sigma_{\text{min}(l_1,n(\rho_1))} \subset Z([0,l_1-1-\text{min}(l_1,n(\rho_1))]_{\rho_1})\ol{\times}(\ol{\Pi}(\mfr{m}^\prime)\rest_{\ol{P_{r-r_1}}})$. Using Proposition \ref{prop prodhomogeneous} and the induction hypothesis for $\ol{\Pi}(\mfr{m}^\prime)$, we get $0\neq \sigma^{(k_{\mfr{m}^\prime})}_{\text{min}(l_1,n(\rho_1))} \subset \sigma^{(k_{\mfr{m}})}$. 
		
		Statement (2) follows from the Lebnitz law as well as the relation between the extended and non-extended Bernstein--Zelevinsky product. More precisely, in the S-cover case we just take the usual Bernstein--Zelevinsky product and then apply Corollary \ref{cor derivativeZ}. In the KP-cover case, we have
		\[(Z(\Delta_1)\ol{\times}\dots\ol{\times} Z(\Delta_k))^{(k_{\mfr m})}=(Z(\Delta_1)^{(k_{\Delta_1})}\ol{\times}\dots\ol{\times} Z(\Delta_k)^{(k_{\Delta_k})}).\]
		Using \cite{zou2022metaplectic}*{Proposition 2.3.(3)} and Corollary \ref{cor derivativeZ} to calculate the total multiplicity of $\Pi(\mfr m)_{\omega}$  (resp. $\Pi(\mfr m^-)_{\omega_-}$ ) in the left-hand (resp. right-hand) side, we get 
		\[ \bigoplus_{\omega}m_{\omega}\cdot \Pi(\mfr m)_{\omega}^{(k_{\mfr m})}= \bigoplus_{\omega_-}m_{\omega_-}\cdot\Pi(\mfr m^-)_{\omega_-}\]
		where $\omega$ and $\omega_-$ in the direct sums range over all possible compatible central characters, and $m_{\omega}$ and $m_{\omega_{-}}$ are related non-negative integers satisfying 
		\[\frac{n^{2k}d_r}{\prod_{i=1}^kd_{r_i}}=\sum_{\omega}m_{\omega}\quad\text{and}\quad \frac{n^{2k}d_{r-k_{\mfr m}}}{\prod_{i=1}^kd_{r_i}}\prod_{i=1}^k\binom{n(\rho_i)}{\min(l_i,n(\rho_i))}=\sum_{\omega_-}m_{\omega_-}.\]
		Taking the socle and comparing the multiplicity, we finish the proof.

		Statement (3) follows from Statement (2) and the fact that socle of the highest derivative of $Z(\mfr m)_{\omega}$ is contained in that of $\Pi(\mfr m)_{\omega}$.
	\end{proof}

	\begin{rmk}
		
		If we introduce the partial order $\leq_e$ on the set of segments as in \cite{lapid2016parabolic}*{\S 3.4} (for two segments not in the same cuspidal line they are simply not comparable), then we may always rearrange $\Delta_i$ such that $\Delta_i\geq_e \Delta_j$ implies that $i\lest j$ and thus the assumptions of Proposition \ref{prop homogeneous} are satisfied under this arrangement.
		
	\end{rmk}
	
	In the special case where the total multiplicity in Proposition  \ref{prop homogeneous}.(2) is one, we have the following result.
	\begin{cor}
		
		Assume that for all $i=1,\dots,k$ we have $l_i\gest n(\rho_i)$, then \[\soc(Z(\mfr m)^{(k_{\mfr m})})=\soc(\Pi(\mfr m)^{(k_{\mfr m})})=Z(\mfr m^-).\]

	\end{cor}

	\subsection{Proof of Theorem \ref{thm highestderivative}}
	From Proposition \ref{prop homogeneous}, we get Statement (1) and (2) of Theorem \ref{thm highestderivative}. Assume the condition of Statement (3), then the total multiplicity $c_{\mfr m
	}$ in Statement (2) is one by direct calculation. Let $Z(\mfr m^-)$ be the socle of $Z(\mfr m)^{(k_{\mfr m})}$. Using the argument in \cite{zelevinsky1980induced}*{\S 8.1} with the aid of Propsition \ref{prop derivativepairing}, we get that the cosocle of $Z(\mfr m)^{(k_{\mfr m})}$ is $Z(\mfr m^-)$ as well. 
	Since the total multiplicity of $Z(\mfr m^-)$ in $Z(\mfr m)^{(k_{\mfr m})}$ is one, we must have $Z(\mfr m)^{(k_{\mfr m})}=Z(\mfr m^-)$. This proves Statement (3) of Theorem \ref{thm highestderivative}.
	
	\subsection{The generic spectrum}
	We record the following classification of generic representations, which is a special case of Theorem \ref{thm highestderivative} but also of independent interest.
	
	\begin{thm}\label{thm genericclassify}
		
		Let $\mfr m=\Delta_1+\dots+\Delta_k$ be a multisegment.
		
		\begin{enumerate}
			\item If $l_i> n(\rho_i)$ for some $i$, then the Whittaker models of both $\Pi(\mfr m)$ and $Z(\mfr m)$ are zero.
			
			\item If $l_i\lest n(\rho_i)$ for every $i$, then the Whittaker model of $Z(\mfr m)$ is non-zero, meaning that $Z(\mfr m)$ is generic.
			
			\item Furthermore, if $l_i= n(\rho_i)$ for every $i$, then the Whittaker dimensions of both  $\Pi(\mfr m)_{}$ and $Z(\mfr m)$ are one.
		\end{enumerate}
		
	\end{thm}
	
	As a direct corollary, we remark here that if $l_i\lest n(\rho_i)$ for every $i$ and $l_j < n(\rho_j)$ for some $j$, then it is possible to have non-uniqueness of the Whittaker models. In particular, the example of $Z(\mfr{m})$ being the theta representations was already studied in detail, both locally and globally, in the work of Kazhdan--Patterson \cite{KP84}.
	
	\section{Wavefront sets and local Langlands correspondence}\label{section WFandLLC}

	\subsection{Wavefront set of $Z(\mfr{m})$}
	We first briefly recall the notion of wavefront sets for genuine representations of a BD-cover $\overline{G}_r$. Let $\mca{N}(G_r)$ denote the partially-ordered set of nilpotent orbits in $\mfr{g}={\rm Lie}(G_r)$ under the conjugation action of $G_r$. Here the partial order is given by the closure ordering in the usual topology of $\mfr{g}$ induced from that of $F$. 
	
	Let $(\pi, V_\pi)$ be a finite-length genuine representation of $\ol{G_r}$. It gives a character distribution $\chi_\pi$ in a neighborhood of $0$ in $\mfr{g}$. Moreover, there exists a compact open subset $S_\pi$ of $0$ such that for every smooth function $f$ with compact support in $S_\pi$, one has (see \cites{How74, HC99, Li12})
	\begin{equation} \label{E:char}
		\chi_\pi(f) = \sum_{\mca{O} \in \mca{N}(G_r)}{c_{\mca{O}, \psi_F}(\pi) }\cdot \int \hat{f}_{\psi_F} \ \mu_\mca{O}.
	\end{equation}
	Here $\mu_\mca{O}$ is a certain Haar measure on $\mca{O}$ properly normalized, and $\hat{f}_{\psi_F}$ is the Fourier transform of $f$ with respect to the Cartan--Killing form on $\mfr{g}$ and the non-trivial character 
	$$\psi_F: F\to \C^\times.$$
	One has $c_\mca{O}(\pi):=c_{\mca{O}, \psi_F}(\pi) \in \C$. 
	We have used  implicitly in \eqref{E:char} an exponent map ${\rm exp}: L \to \overline{G}_r$ defined for a sufficiently small open set $L \subset \mfr{g}$ containing $0$; it is used to ``pull-back" the character distribution of $\pi$ defined on $\overline{G}_r$ to be on (a small neighborhood of 0 in) $\mfr{g}$, see \cite{Li12}*{\S 4.3}. 
	
	Denote
	$$\mca{N}(\pi) = \set{\mca{O} \in \mca{N}(G_r): \ c_\mca{O}(\pi) \ne 0},$$
	and let
	$${\rm WF}(\pi) \subseteq \mca{N}(\pi)$$
	be the subset consisting of all maximal elements in $\mca{N}(\pi)$. 
	Here, ${\rm WF}(\pi)$ is called the wavefront set of $\pi$. For $G_r$ (i.e., in the $n=1$ case) and irreducible $\pi$, the set $\WF(\pi)$ is singleton from the work of M\oe glin--Waldspurger \cite{moeglin1987modeles}, and it can be computed from the mulltisegment $\mfr{m}$ underlying $\pi=Z(\mfr{m})$. The goal of this section is to extend such results on $G_r$ to KP-covers or Savin covers.
	
	For this purpose, we briefly recall the notion of degenerate Whittaker models and their relation with wavefront set, see \cite{moeglin1987modeles, gomez2017generalized} for details.
	Let $\mfr{p}:=\mfr{p}_\mca{O}$ be a partition of $r$, which is associated with a nilpotent orbit $\mca{O}$ of $G_r$. 
	First, the nilpotent orbit $\mfr{p}$ corresponds to a $\mfr{sl}_2$-triple, and this gives rise to a unipotent subgroup $U_\mfr{p} \subseteq G_r$ which surjects naturally to a Heisenberg group $H_\mfr{p}$ as in 
	$$U_\mfr{p} \onto H_\mfr{p},$$
	both of which depend on $\mfr{p}$. Moreover, the orbit $\mfr{p}$ also gives the center $C_\mfr{p}$ of $H_\mfr{p}$. Using $\psi_F$ (and the Killing form on $\mfr{g}$), one has a natural non-trivial character $\psi_\mfr{p}: C_\mfr{p} \to \C^\times$.
	Let $\omega_{\psi_\mfr{p}}$ be the irreducible representation of $H_\mfr{p}$ with central character $\psi_\mfr{p}$. By pull-back, we have a representation $\omega_{\psi_\mfr{p}}^*$ of $U_\mfr{p}$. For a finite-length genuine representation $\pi$ of $\ol{G_r}$, the degenerate Whittaker model of $\pi$ is then by definition
	$$\Wh(\pi; \mfr{p}):= \Hom_{\overline{G}_r}(\pi, {\rm Ind}_{\mu_n \times U_\mfr{p}}^{\overline{G}_r} \epsilon \boxtimes \omega_{\psi_\mfr{p}}^*) = \Hom_{\mu_n \times U_\mfr{p}}(\pi, \epsilon \boxtimes \omega_{\psi_\mfr{p}}^*).$$
	We also define
	$$\mca{N}_{\rm{deg}}(\pi):=\max \set{\mca{O}: \ \Wh(\pi; \mfr{p}_\mca{O}) \ne 0} \subseteq \mca{N}(G_r).$$
	It follows from the work \cites{moeglin1987modeles, Var14, patel2014theorem} that  the equality 
	\begin{equation} \label{E:W=W}
		\WF(\pi) = \mca{N}_{\rm{deg}}(\pi)
	\end{equation}
	holds; moreover, with the involved measures properly normalized, for every $\mca{O}$ in the above (equal) sets, one has $c_\mca{O}(\pi) = \dim_\C \Wh(\pi; \mfr{p}_\mca{O})$.
	Furthermore, by \cite[Theorem A]{gomez2017generalized} or \cite[Theorem 5]{GGS21} (the latter already stated for covering groups), we have an injection
	\begin{equation}\label{E:injsemiWh}
		\Wh_{\mfr{p}}(\pi)\simeq\Wh(J_{U_\mfr{p}}(\pi)) \into \Wh(\pi; \mfr{p}),
	\end{equation}
	where  
	the left-hand side represents the $\mfr{p}$-semi-Whittaker model of $\pi$.   
	%We mention in passing that the work 
	
	By \cite{patel2014theorem}*{Theorem 2}, the embedding \eqref{E:injsemiWh} becomes an isomorphism for those orbits $\mfr{p}$ that lie in the wavefront set of $\pi$. As a result, if we similarly define
	\[\mca{N}_{\rm{sem}}(\pi):=\max \set{\mca{O}: \ \Wh_{\mfr{p}_\mca{O}}(\pi) \ne 0} \subseteq \mca{N}(G_r),\]
	then we have
	\begin{equation}\label{E:WHdWHsd}
		\mca{N}_{\rm{sem}}(\pi)=\mca{N}_{\rm{deg}}(\pi)=\WF(\pi).
	\end{equation}
	
	Before stating the main theorem, we need the following lemma for general BD-covers. Write $\beta=(r_1,\dots,r_k)$ for a composition of $r$ and $M_\beta=\prod_{i=1}^{k} G_{r_i}$ for the Levi subgroup of $G_r$.
	
	\begin{lm}\label{lem: BDWFtensorprod}
		
		For an $n$-fold BD-cover $\ol{G_r}$ and $\pi_i \in \Irrg(\ol{G_{r_i}})$ with $i=1,\dots,k$ and $r=r_1+\dots+r_k$ and any irreducible constituent $\pi$ of the EMTP $\ol{\boxtimes_{i=1}^k} \pi_i$, we have
		\[\mca{N}_{\rm{deg}}( \pi) = \prod_{i=1}^k \mca{N}_{\rm{deg}}(\pi_i).\]
		
	\end{lm}
	
	\begin{proof}
		
		Let $\mfr{p}_i, 1\lest i \lest k$ be a partition of $r_i$. For $\pi_i \in \Irrg(\ol{G_{r_i}})$, let $\Wh(\pi_i; \mfr{p}_i)$ be the degenerate Whittaker model of $\pi_i$. 
		Now we set $\mfr{p}:=(\mfr{p}_1, ..., \mfr{p}_k)$ and view it as an orbit for $M_\beta$.
		For any $\pi \in \Irrg(\overline{M_\beta})$ and every character $\chi$ of $M_\beta$, it is easy to see 
		$$\Wh(\pi \cdot \chi; \mfr{p}) \simeq \Wh(\pi; \mfr{p}).$$
		Applying this to any irreducible constituent $\pi$ of  $\ol{\boxtimes_{i=1}^k} \pi_i$, it follows from Corollary \ref{cor MTP}(2) that 
		\begin{equation} \label{E:red01}
			\Wh(\ol{\boxtimes_{i=1}^k} \pi_i; \mfr{p}) \simeq c_1 \cdot \Wh(\pi; \mfr{p})
		\end{equation}
		as vector spaces, for some $c_1 \in \N_{\gest 1}$. On the other hand, if we write 
		$$\pi_\flat:= \pi_1|_{\ol{G_{r_1}^{(n)}}} \boxtimes ... \boxtimes \pi_k|_{\ol{G_{r_k}^{(n)}}}$$
		as a genuine representation of $\ol{M_\beta^{[n]}}:=\ol{\prod_{i=1}^kG_{r_i}^{(n)}}$, then by definition $\ol{\boxtimes_{i=1}^k} \pi_i = {\rm Ind}_{\ol{M_\beta^{[n]}}}^{\ol{M_\beta}} \pi_\flat$. 
		We get by Mackey theory that
		$$(\ol{\boxtimes_{i=1}^k} \pi_i)|_{\ol{M_\beta^{[n]}}} = \bigoplus_{s \in M_\beta/M_\beta^{[n]}} {}^s \pi_\flat.$$
		We can write $s=(s_1, s_2, ..., s_i, ..., s_k)$ with $s_i \in G_{r_i}/G_{r_i}^{(n)}$. Then 
		$${}^s \pi_\flat:= ({}^{s_1} \pi_1)|_{\ol{G_{r_i}^{(n)}}} \boxtimes ... \boxtimes ({}^{s_k}\pi_k)|_{\ol{G_{r_k}^{(n)}}}.$$
		We have 
		$$ \Wh(({}^{s_i} \pi_i)|_{\ol{G_{r_i}^{(n)}}}; \mfr{p}_i)= \Wh({}^{s_i} \pi_i; \mfr{p}_i) \simeq \Wh(\pi_i; \mfr{p}_i).$$
		This gives
		$$\Wh({}^s \pi_\flat; \mfr{p}) \simeq \prod_{i=1}^k \Wh(\pi_i; \mfr{p}_i)$$
		and hence
		$$ \Wh(\ol{\boxtimes_{i=1}^k} \pi_i; \mfr{p}) = \Wh((\ol{\boxtimes_{i=1}^k} \pi_i)|_{\ol{M_\beta^{[n]}}}; \mfr{p})   \simeq \bigoplus_{s \in M_\beta/M_\beta^{[n]}} \Wh({}^s \pi_\flat; \mfr{p})  \simeq \val{M_\beta/M_\beta^{[n]}} \cdot \prod_i \Wh(\pi_i; \mfr{p}_i).$$
		This coupled with \eqref{E:red01} shows that $\mfr{p}=(\mfr{p}_i)_i$ supports the degenerate Whittaker model of $\pi$ if and only if $\mfr{p}_i$ supports degenerate Whittaker model of $\pi_i$ for every $i$. 
		
	\end{proof}
	
	Now we may state and prove the main theorem.
	
	\begin{thm} \label{T:WF}
		Let $\ol{G_r}$ be an $n$-fold Kazhdan--Patterson cover or a Savin cover. Assume $p\nmid n$. Let $Z(\mfr{m}) \in \Irrg(\overline{G_r})$ be associated with a multisegment $\mfr{m}$ and let $\lambda_\mfr{m}$ be the partition given as in \eqref{D:lambda}. Then
		\begin{equation} \label{E:WF}
			{\rm WF}(Z(\mfr{m}))=\set{\lambda_\mfr{m}}.
		\end{equation}
		In particular, the wavefront set of $Z(\mfr{m})$ is a singleton set. 
	\end{thm}
	\begin{proof}
		
		In view of the equality \eqref{E:W=W}, it suffices to show 
		\begin{equation} \label{E:Wh}
			\mca{N}_{\rm{deg}}(Z(\mfr{m})) = \set{\lambda_\mfr{m}}.
		\end{equation}
		Note it follows from 
		Corollary \ref{cor generalizedWhitt} that $\Wh(J_{U_{\lambda_\mfr{m}}}(Z(\mfr{m})))\ne 0$. Applying this and \eqref{E:injsemiWh} to $\mfr{p}:=\lambda_\mfr{m}$ and $\pi:=Z(\mfr{m})$, we get $\Wh(Z(\mfr{m}); \lambda_\mfr{m})\ne 0$, i.e., $\lambda_\mfr{m}$ supports degenerate Whittaker model of $Z(\mfr{m})$.
		
		To prove \eqref{E:Wh}, it now suffices to show that $\mca{N}_{\rm{deg}}(Z(\mfr{m})) \lest \lambda_\mfr{m}$. Furthermore, we will show that
		\begin{equation} \label{E:W-Pi}
			\WF(\Pi(\mfr{m})) = \mca{N}_{\rm{deg}}(\Pi(\mfr{m})) = \set{\lambda_\mfr{m}},
		\end{equation}
		which implies $\WF(Z(\mfr{m})) = \mca{N}_{\rm{deg}}(Z(\mfr{m})) \lest \lambda_\mfr{m}$.
		
		Now, by a result of M\oe glin--Waldspurger \cite{moeglin1987modeles}*{II.1}, which equally holds for covering groups, we have
		\begin{equation} \label{E:O-ind}
			{\rm WF}(\Pi(\mfr{m})) = {\rm Ind}_{\ol{M_\beta}}^{\ol{G_r}} {\rm WF}( \boxtimes_{i=1}^k Z(\Delta_i) ),
		\end{equation}
		where $\boxtimes_{i=1}^k Z(\Delta_i)$ is the metaplectic tensor product of $Z(\Delta_i) \in \Irrg(\ol{G_{r_i}})$.
		Using Lemma \ref{lem: BDWFtensorprod} for $\pi_i = Z(\Delta_i)$, we get 
		\[\mca{N}_{\rm{deg}}( \boxtimes_{i=1}^k Z(\Delta_i)) = \prod_{i} \mca{N}_{\rm{deg}}(Z(\Delta_i)),
		\]
		or equivalently
		\begin{equation*}\label{E:WF-eq}{\rm WF}( \boxtimes_{i=1}^k Z(\Delta_i)) = \prod_{i} {\rm WF}(Z(\Delta_i)).
		\end{equation*}

		To proceed for \eqref{E:W-Pi}, for every $1\lest i \lest k$, we consider $a_i \gest 0, 0\lest b_i < n(\rho_i)$ such that
		$$l_i= a_i \cdot n(\rho_i) + b_i.$$
		
		By Corollary \ref{cor generalizedWhittseg}, the partition $((r_i^0\cdot n(\rho_i))^{a_i}, r_i^0 b_i)$ corresponds to the largest nilpotent orbit that contributes to a semi-Whittaker model of $Z(\Delta_i)$. Thus using \eqref{E:WHdWHsd} we have
		$${\rm WF}(Z(\Delta_i)) = ((r_i^0\cdot n(\rho_i))^{a_i}, r_i^0 b_i).$$
		It can be checked easily that this partition is exactly $\lambda_{\Delta_i}$, i.e., the equality \eqref{E:WF} holds for the single segment $\mfr{m} = \Delta_1$ case.
		From the definition of $\lambda_\mfr{m}$ it is easy to see that
		$$\lambda_\mfr{m} = \sum_{i=1}^k \lambda_{\Delta_i} = \Ind_{M_\beta}^{G_r} (\lambda_{\Delta_1}, ..., \lambda_{\Delta_k}),$$
		where the right-hand side means the induction of orbits. In view of \eqref{E:O-ind}, this gives \eqref{E:W-Pi} as desired. The proof is completed.
	\end{proof}
	
	\subsection{Relation with local Langlands correspondence}
	In this subsection, for Kazhdan--Patterson covers we reformulate Theorem \ref{T:WF} in the framework of \cite{GLLS} using a version of the hypothetical local Langlands correspondence (LLC) for such $\ol{G_r}$, and also the covering Barbasch--Vogan duality.
	
	Let $\ol{G_r}$ be a Kazhdan--Patterson cover or a Savin cover. Let $\ol{G_r}^\vee$ be the complex dual group of $\ol{G_r}$. One has
	$$\ol{G_r}^\vee \simeq
	\begin{cases}
		\set{(g, \gamma) \in G_r(\C) \times G_1(\C): \det(g)=\gamma^{d_r}} & \text{ for KP-covers},\\
		G_r(\C) & \text{ for S-covers}.
	\end{cases}
	$$
	Clearly we have a natural group homomorphism
	$$\eta: \ol{G_r}^\vee \longrightarrow G_r(\C),$$
	which is the identity for S-covers and is the projection on the first coordinate for KP-covers.
	
	For every partition $\mfr{p}^\vee=(p_1, ..., p_k)$ of $r$, viewed as a nilpotent orbit of $\ol{G_r}^\vee$, we define
	$$d_{BV, G_r}^{(n)}(\mfr{p}^\vee)=\sum_{i=1}^k \mfr{s}(p_i; n_\alpha)$$
	with the partition $\mfr{s}(p_i; n_\alpha)$ given by 
	\begin{equation}
		\mfr{s}(p_i;n_\alpha):=(n_\alpha^{a_i}, b_i)=(\underbrace{n_\alpha, n_\alpha, ..., n_\alpha}_{a_i}, b_i),
	\end{equation}
	where $p_i=a_i \cdot n_\alpha + b_i, 0\lest b_i < n_\alpha$. Here, the summation of partition is the usual one, for example, $(5,4,2,2) + (6,3)+ (5,2,2)=(16,9,4,2)$. We view $d_{BV, G_r}^{(n)}(\mfr{p}^\vee)$ as a nilpotent orbit for $\mbf{G}_r(F^{\rm alg})$. This gives the covering Barbasch--Vogan duality for $\ol{G}_r$:
	$$d_{BV, G_r}^{(n)}: \mca{N}(\ol{G_r}^\vee) \longrightarrow \mca{N}(\mbf{G}_r(F^{\rm alg})).$$
	Note that $n_\alpha=n$ for KP-covers, and $n_\alpha=n/\gcd(n,2)$ for S-covers.
	
	For $G_r$, we know the map $\mca{N}(G_r)\rightarrow\mca{N}(\mbf{G}_r(F^{\rm alg}))$ given by $\mathcal{O} \mapsto \mathcal{O}\otimes F^{\rm alg}$ is a bijection and thus we identify the two sets. 
	From the prior work \cite{GLLS, GW}, it is expected that one has
	$${\rm WF}(Z(\mfr{m})) = \set{d_{BV, G_r}^{(n)}(\mca{O}(\phi_{{\rm AZ}(Z(\mfr{m}))}))},$$
	where ${\rm AZ}$ is the Aubert--Zelevinsky involution discussed as in \cite{GLLS}*{\S 2.3}, and $\phi_\pi$ is the hypothetical L-parameter associated with $\pi$.
	
	To show that Theorem \ref{T:WF} is compatible with the above speculation, in the remaining of this subsection, we consider Kazhdan--Patterson covers exclusively and utilize a version of LLC that arises from the metaplectic correspondence for discrete series given by Flicker, Kazhdan and Patterson \cite{KP84, kazhdan1986towards, flicker1986metaplectic}. Indeed, from loc. cit. we see that every genuine supercuspidal representation $\rho \in \Irrg(\ol{G_r})$ can be lifted to a certain discrete series ${\rm MC}(\rho) \in \Irr(G_r)$. Then it is natural to consider
	$$\phi_\rho:=\phi_{{\rm MC}(\rho)}: \W_F \times \SL_2(\C) \longrightarrow G_r(\C).$$
	
	Let $Z(\mfr{m}) \in \Irr_\epsilon(\ol{G_r})$ be associated with $\mfr{m}=(\Delta_1, ..., \Delta_k), \Delta_i = [0, l_i-1]_{\rho_i}$ and $\rho_i$ being a supercuspidal representation. Let $L(\mfr{m})$ be the unique irreducible quotient associated with $\mfr{m}$ as in \cite{kaplan2022classification}. Then we know $Z(\mfr{m}) = {\rm AZ}(L(\mfr{m}))$, which follows from the same argument as in \cite[Proposition A.7]{MS13}.  Also, the metaplectic lifting of $\rho_i$ is of the form
	$${\rm MC}(\rho_i) = L(\Delta_i^\sharp), \text{ with } \Delta_i^\sharp = [0, l(\rho_i)-1]_{\rho_i^\sharp},$$
	with $\rho_i^\sharp \in {\rm Cusp}(G_{r_i^\sharp}), r_i^\sharp:=r_i^0/l(\rho_i)$. 
	
	For each $i$, let $\phi_{\rho_i^\sharp}$ be the L-parameter of $\rho_i^\sharp$. We call the homomorphism
	$$\phi_{L(\mfr{m})}: W_F \times \SL_2(\C) \longrightarrow G_r(\C)$$
	given by 
	$$\phi_{L(\mfr{m})}:= \bigoplus_{i=1}^k \phi_{\rho_i^\sharp}  \val{\cdot}^{\frac{l_i l(\rho_i)-1}{2}} \boxtimes S_{l_i \cdot l(\rho_i)}$$
	the L-parameter of $L(\mfr{m})$. A caveat is that it is not valued in $\ol{G_r}^\vee$ by definition. We will show that the parameter $\phi_{L(\mfr{m})}$ defined above can be lifted through $\eta$ naturally and thus is valued in $\ol{G_r}^\vee$.
	
	First, consider the torus $\mbf{T}_z \subseteq Z(\mbf{G}_r)$ associated with the lattice $c_{n,r} \cdot e_c$, where $c_{n,r}:=n/d_r$ and $e_c:=\sum_{i=1}^r e_i$. It gives the the group $T_z$ of $F$-points, and we know
	$$Z(\ol{G_r}) \simeq \ol{T_z}.$$
	Let $\omega: Z(\ol{G_r}) \to \C^\times$ be a genuine character.
	By choosing a distinguished genuine character $\chi_\psi: Z(\ol{T}) \to \C^\times$ as in \cite[\S 7]{gan2018langlands}, viewed as a character of $Z(\ol{G_r})$ by restriction, we get a linear character
	$$\omega^\psi:=\omega/\chi_\psi: T_z \to \C^\times.$$
	
	Consider $L(\mfr{m}) \in \Irrg(\ol{G_r})$, which has a central character $\omega_\mfr{m}$ and thus gives
	$$\omega_\mfr{m}^\psi: T_z \to \C^\times.$$
	The local Langlands correspondence for $T_z$ gives the L-parameter 
	$$\phi_{\omega_\mfr{m}^\psi}: W_F \to \C^\times$$
	of $\omega_\mfr{m}^\psi$, which is explicitly given by $\phi_{\omega_\mfr{m}^\psi}(a)=\omega_\mfr{m}^\psi((c_{n,r}\cdot e_c)\otimes {\rm rec}(a))$, where ${\rm rec}:W_F\rightarrow F^{\times}$ is the Artin reciprocity for $F$.
	Now we consider the map
	$$\tilde{\phi}_{L(\mfr{m})}: W_F \times \SL_2(\C) \longrightarrow G_r(\C) \times G_1(\C)$$
	given by
	$$\tilde{\phi}_{L(\mfr{m})}(a, g):=(\phi_{L(\mfr{m})}(a, g), \phi_{\omega_\mfr{m}^\psi}(a)).$$
	
	\begin{lm}
		The parameter $\tilde{\phi}_{L(\mfr{m})}$ is valued in $\ol{G_r}^\vee$.
	\end{lm}
	\begin{proof}
		To save notation, we identify $a$ with ${\rm rec}(a)$, and also write $a^y:=y\otimes a$ for $y\in Y_{G_r}$ and $a$. Write $\val{a}=\val{{\rm rec}(a)}=q^{-v_a}$ for some $v_a\in\Z$.
		
		In view of the definition of $\ol{G_r}^\vee$ for KP covers, it suffices to show that 
		\begin{equation} \label{E:eq}
			\det(\phi_{L(\mfr{m})}(a, g)) = (\phi_{\omega_\mfr{m}^\psi}(a))^{d_r}.
		\end{equation}
		By the definition of $\phi_{L(\mfr{m})}$, the left-hand side of \eqref{E:eq} is equal to
		$$\prod_{i=1}^k \det(\phi_{\rho_i^\sharp}(a))^{l_i \cdot l(\rho_i)} \cdot q^{-v_ar_i^0l_i \cdot \frac{l_i l(\rho_i)-1}{2}}.$$
		To compute the right-hand side, we first recall that there is a natural section $\mbf{s}: T \to \ol{T}$, which is actually a splitting when restricted to $(nY_{G_r})\otimes F^\times \subseteq T$. Also, one property of $\chi_\psi$ (see \cite[\S 7]{gan2018langlands}) is that $\chi_\psi(\mbf{s}(a^{ny}))=1$ for all $y\in Y_{G_r}$ and $a$. This gives that
		$$(\phi_{\omega_\mfr{m}^\psi}(a))^{d_r} = \omega_\mfr{m}^\psi(a^{ne_c}) = (\omega_\mfr{m}/\chi_\psi)(a^{ne_c}) = \omega_\mfr{m}(\mbf{s}(a^{ne_c})).$$
		For every $i$, let $e_{c,i} \in Y_{G_{r_i^0}}$ be the natural generator of the cocharacter lattice of $Z(G_{r_i^0})$. Since $L(\mfr{m})$ is a constituent of the induced representation
		$$\bigtimes_{i,j} \rho_i \val{\cdot}^{j/n(\rho_i)}$$
		with $1\lest i \lest k$ and $0\lest j \lest l_i-1$, we see that 
		$$\begin{aligned}
			\omega_\mfr{m}(\mbf{s}(a^{ne_c})) = &  \prod_{i=1}^k \prod_{j=0}^{l_i-1} \omega_{\rho_i}(\mbf{s}(a^{ne_{c,i}})) \cdot \val{\mbf{s}(a^{ne_{c,i}})}^{j/n(\rho_i)} \\
			=& \prod_{i=1}^k  \omega_{\rho_i}(\mbf{s}(a^{ne_{c,i}}))^{l_i} \cdot q^{-v_ar_i^0l(\rho_i) \cdot \frac{l_i(l_i-1)}{2}},
		\end{aligned} $$
		where $\omega_{\rho_i}$ means the central character of $\rho_i$. However, the property of metaplectic correspondence shows that (see \cite[Pg. 95-96]{flicker1986metaplectic})
		$$\omega_{\rho_i}(\mbf{s}(a^{ne_{c,i}})) = \omega_{{\rm MC}(\rho_i)}(a^{e_{c,i}}),$$
		where ${\rm MC}(\rho_i) = L([\rho_i^\sharp, \rho_i^\sharp\val{\cdot}, ..., \rho_i^\sharp\val{\cdot}^{l(\rho_i)-1}]) \in \Irr(G_{r_i^0})$. Also, it is easy to see that
		$$\omega_{{\rm MC}(\rho_i)}(a^{e_{c,i}}) = \det(\phi_{\rho_i^\sharp})^{l(\rho_i)} \cdot q^{-v_ar_i^\sharp \cdot \frac{l(\rho_i)(l(\rho_i)-1)}{2}}.$$
		Now, using the above three displayed equalities, one can simplify $ \omega_\mfr{m}(\mbf{s}(a^{ne_c}))$ and show that \eqref{E:eq} holds. This completes the proof.
	\end{proof}
	
	Thus we have a well-defined map $\tilde{\phi}_{L(\mfr{m})}: \W_F \times \SL_2(\C) \to \ol{G_r}^\vee$, which is called the L-parameter of $L(\mfr{m})$.
	
	\begin{cor} \label{cor WFZmBVdual}
		Keep notation as above. We have
		$${\rm WF}(Z(\mfr{m})) = \set{d_{BV, G_r}^{(n)}(\mca{O}(\tilde{\phi}_{L(\mfr{m})}))}.$$
	\end{cor}
	\begin{proof}
		We have
		$$\mca{O}(\tilde{\phi}_{L(\mfr{m})}) = ((l_1 \cdot l(\rho_1))^{r_1^\sharp}, ..., (l_i \cdot l(\rho_i))^{r_i^\sharp}, ..., (l_k \cdot l(\rho_k))^{r_k^\sharp}).$$
		Now a direct computation shows that $\lambda_\mfr{m} = d_{BV, G_r}^{(n)}(\mca{O}(\tilde{\phi}_{L(\mfr{m})}))$. Indeed, if we write 
		$$l_i l(\rho_i) = n \alpha_i + \beta_i$$
		with $\alpha_i \gest 0$ and $0\lest \beta_i <n$. We get $\alpha_i \gest 1$ if and only if $l_i l(\rho_i) \gest n$, equivalently, $l_i \gest n(\rho_i)$. 
		Then the leading numeral of $d_{BV, G_r}^{(n)}(\mca{O}(\tilde{\phi}_{L(\mfr{m})}))$ is equal to
		$$\sum_{i=1}^k r_i^\sharp \cdot \min (n, l_i l(\rho_i)) = \sum_{i=1}^k r_i^\sharp \cdot l(\rho_i) \cdot \min (n(\rho_i), l_i) =\sum_{i=1}^k r_i^0 \cdot \min (n(\rho_i), l_i),$$
		which is exactly $k_1$. The other numerals in $\lambda_\mfr{m}$ can be checked in a similar way.
	\end{proof}
	
	\subsection{Some remarks}
	For a partition $\mfr{p}=(p_1, p_2, ..., p_s)$, we write
	$${\rm H}(\mfr{p}):=s\quad \text{and}\quad {\rm W}(\mfr{p}):={\rm H}(\mfr{p}^\top)=\max\set{p_i: 1\lest i \lest s}.$$
	We can view ${\rm H}(\mfr{p})$ and ${\rm W}(\mfr{p})$ as the height and width (and thus the ``shape'') of the Young diagram associated with $\mfr{p}$. Let $\pi\in \Irr(G_r)$ and consider the L-parameter $\phi_{{\rm AZ}(\pi)}$ of ${\rm AZ}(\pi)$. 
	We get
	\begin{equation} \label{E:H}
		{\rm H}(\mca{O}(\phi_{{\rm AZ}(\pi)})) = \max\set{k: \ \pi^{(k)} \ne 0}.
	\end{equation}
	On the other hand, if we assume that $\pi$ can always be ``lifted" to a genuine representation $\pi_n$ of KP-cover or S-cover, for every $n$; then from the above disucssion, we expect that 
	\begin{equation} \label{E:W}
		{\rm W}(\mca{O}(\phi_{{\rm AZ}(\pi)})) = \min \set{n: \pi_n \text{ is generic}}.
	\end{equation}
	The relations \eqref{E:H} and \eqref{E:W} above can be viewed as conservation-alike, and they hold for Iwahori-spherical representations $\pi$, whenever the parameter $\phi_{{\rm AZ}(\pi)}: \WD_F \to G_r(\C)$ can be lifted through the natural homomorphism $\eta: \ol{G_r}^\vee \to G_r(\C)$ for a KP-cover or S-cover $\ol{G_r}$.
	
	%\section{Declarations}
	
	%\begin{itemize}
	%\item \textbf{Conflict of Interest:} On behalf of all authors, the corresponding author states that there is no conflict of interest.
	
	%\item \textbf{Availabililty of data and material:} Not applicable.
	
	%\item \textbf{Ethical Approval:} Not applicable.
	
	%\item \textbf{Authors' contributions:} Equally contributed.
	%\end{itemize}

	%%%%%%%%%%%%%
	
	\bibliographystyle{alpha}
	\bibliography{mybib}
	
	%\begin{bibdiv}
	%	\begin{biblist}[\resetbiblist{9999999}]*{labels={alphabetic}}
		
		%	\end{biblist}
	%\end{bibdiv}

\end{document}